\documentclass[11pt,a4paper]{article}
\usepackage[utf8]{inputenc}
\usepackage[T1]{fontenc}
\usepackage{amsmath,amssymb,amsfonts}
\usepackage{amsthm}
\usepackage{graphicx}
\usepackage{geometry}
\usepackage{enumitem}
\usepackage{bm}
\usepackage{xcolor}
\usepackage{hyperref}
\usepackage{array}
\usepackage{multirow}
\usepackage{booktabs}
\usepackage{float}
\usepackage{tikz}
\usepackage{caption}
\usepackage{subcaption}

\newtheorem{theorem}{Theorem}[section]
\newtheorem{lemma}[theorem]{Lemma}
\newtheorem{corollary}[theorem]{Corollary}
\newtheorem{proposition}[theorem]{Proposition}
\theoremstyle{definition}
\newtheorem{definition}[theorem]{Definition}
\newtheorem{remark}[theorem]{Remark}

\newcommand{\curl}{\mathbf{curl}}

\title{Error Analysis for a Nonconforming Finite Element Discretization of the 3D-CCPF Model}
\author{{Cl\'ement Adand\'e}$^{1}$ \quad {Koffi W. Hou\'edanou}$^{2}$ \quad {Guy A. D\`egla}$^{1}$
	\\
	\small{	$^1$ Institut de Mathématiques et de Sciences Physiques (IMSP) \& $^2$ Faculté des Sciences et Techniques (FAST) }\\ \small{Université d'Abomey-Calavi (UAC), Bénin}
 }

\date{}

\begin{document}
	
	\maketitle
	
	\begin{abstract}
		This paper presents an a-priori and a-posteriori error analysis for the time-dependent case of a coupled continuum pipe-flow in three dimensions on anisotropic meshes (i.e. elements with very large aspect ratio). Our investigation includes nonconforming discretizations as well as different elements and different versions of Crouzeix-Raviart finite element methods. 
		Under suitable regularity assumptions on the exact
		solution, we derive optimal a priori error estimates in the $H^1$-like norm.
		Leveraging residual elements, local error indicators and a global estimator are generated, demonstrating reliability and efficiency. The efficiency of error estimators holds unconditional, though its reliability involves the alignment measure notion. Mesh requirements are established to ensure the optimality of the error estimator.
	\end{abstract}
	
	\section{Introduction}\label{sec1}
	
	Karst aquifers are found almost everywhere in the world. A karst system is a heterogeneous, discontinuous system comprising highly conductive conduits, more or less well interconnected in a generally three-dimensional network, located in a matrix of generally carbonate rocks that may themselves have conductive and capacitive properties \cite{bonnet1977introduction}. These systems are characterized by caves and networks of conduits that allow water to infiltrate rapidly and circulate through the rocks. The difficulty of characterizing this network complicates modeling \cite{bonnet1977introduction}. Moreover, C. Bernard affirms on this subject that they stubbornly resist attempts at modeling \cite{collignon2018aquifere}. They are distinguished from other aquifers by their high permeability, chemical reactivity and vulnerability to contamination.
	
	Several models attempt to simulate water flow and solute transport in karst aquifers. Significant mathematically rigorous progress has been achieved via the so-called coupled (Navier) Stokes-Darcy model with classical Beavers-Joseph interface boundary condition or various simplified interface conditions. However geologists have proposed various ad-hoc simplified models in order to make progress. One of the most popular models is the so-called coupled-continuum pipe flow (CCPF) model where the conduits are simplified into a network of one dimensional pipes \cite{wang2010coupled}. This original CCPF model, both the steady state case and the time-dependent case, are in fact ill-posed \cite{wang2010coupled}. In \cite{hua2009modeling}, Hua attempts to solve this problem and provides another formulation of this model. This model is well-posed in 2D but not in 3D. The ill-posedness of original and Hua's 3D model is due to the fluid exchange occurring on a very singular space: point singularity in the original CCPF and line singularity in Hua's model. This is the consequence of the fact that on the physical side, these simplified models do not reflect the physical reality of the fluid exchange happening over the interface between the conduit and matrix. X. Wang \cite{wang2010coupled} has proposed a new modification model of the CCPF model that is physically relevant appealing and mathematically sound 3D-CCPF model. Taking into consideration the physical observation of exchange of fluids over the whole interface \(\Gamma\) and retaining the Barenblatt type fluid exchange relation, X. Wang arrives at the first equation in new model $(\ref{model1})$. On the other hand, the conduit flow is not solved completely, and hence it makes sense to use the averaged head of the matrix on a proper cross section of the interface to calculate the Barenblatt exchange term in the conduit so that the conduit equation remains one dimensional. This leads to the second equation of (\ref{model1}) \cite{wang2010coupled}. In other terms, the new 3D-CCPF model of our concern in this work is given by $(\ref{model1})$ for the simple case of a one dimensional conduit centered at the \(x\) axis and laminar flow:
	
	\begin{equation}\label{model1}
		\left\{
		\begin{array}{ll}
			S\dfrac{\partial u^m}{\partial t} -\nabla \cdot \left(\mathbb{K}\nabla u^m\right) = -\dfrac{\alpha_{ex}}{|\Gamma_x|}\left(u^m -u^c\right)\delta_{\Gamma} + R^m & \text{in } \Omega_m\\[10pt]
			-\dfrac{\partial}{\partial x}\left(D\dfrac{\partial u^c}{\partial x}\right) = \alpha_{ex}\left(\dfrac{1}{|\Gamma_x|}\int_{\Gamma_x}u^m dl_x - u^c\right) + R^c & \text{in } \Omega_c
		\end{array}
		\right. \quad 
	\end{equation}
	
	where \(u^m\) represents hydraulic head in porous matrix \(\Omega_{m}\) and \(u^{c}\) hydraulic head in pipe conduit \(\Gamma\) (considered as a one dimension curve \(\Omega_{c}\)), and \(\Gamma_{x}\) is a cross section of \(\Gamma\) at \(x\), \(R^m\) and \(R^c\) are source terms. The constants \(\alpha_{ex}\), \(K\), and \(D\) are positives and represent respectively the exchange coefficient, the hydraulic conductivity and the laminar Poiseuille constant. This work aims to perform an a-posteriori error analysis on this model based on a nonconforming finite element discretization.
	
	Less attention has been paid to the $3D$-CCPF model in this axis. Several researches have been achieved about the 2D-CCPF model \cite{houedanou2019posteriori, liu2015anisotropic, liu2014finite, liu2016new, liu2014anisotropic}. In \cite{houedanou2019posteriori}, the author performs an a-posteriori error analysis using a unified anisotropic finite element discretization and covers two dimensional domains with conforming and nonconforming elements. Lower and upper error bounds are made under some assumptions on the elements. In \cite{liu2015anisotropic}, Wilson element on anisotropic mesh is used to solve the Darcy equation on porous matrix and an optimal error estimates in \(L^2\) and \(H^1\) norms are established independent of the regularity on the mesh. In \cite{liu2014finite}, the authors applied a finite volume element method to approach the continuum pipe-flow/Darcy problem and provide optimal error estimates in \(L^2\) and \(H^1\). The same tasks are performed in \cite{liu2016new} on the model using the new nonconforming finite element on the porous matrix which are better than \(Q_{1,1}\) conforming element and Wilson nonconforming element on the same mesh.
	
	To our knowledge, there are no existing results on the approximation of the 3D-CCPF model using the finite element method on either isotropic or anisotropic discretizations. This work and the next will cover this research axis. We intend to make an a-priori and an a-posteriori error analysis by using nonconforming finite element discretization on isotropic and anisotropic meshes; some numerical tests using Crouzeix-Raviart finite element method on tetrahedra, on pentahedra and on hexahedra; develop a new CCPF model in 3D which taking account two pipe conduits.
	
	The remaining of the paper is described as follow. Section \ref{sec2} provides some notations used in the following. Section \ref{sec3} presents the discrete problems. Section \ref{sec4} exposes some analytical tools. Section \ref{sec5} and Section \ref{sec6} present respectively an a-priori and a-posteriori error estimation for the time and the space discretization. The Section \ref{conclusion} concludes the results.
	
	\section{Notation}\label{sec2}
	\subsection{Notation relative to spaces and norms}\label{sec21}
	Let \(\mathbb{P}^k\) be the space of polynomials of order \(k\) or less. Instead of \(x\leq c y\) or \(c_{1}x\leq y\leq c_{2}x\) (with constants independent of \(x,y\) and meshes), we use the shorthand notation \(x\lesssim y\) and \(x\sim y\) respectively.
	
	Let \(X\) be a separable Banach space equipped with the norm \(\| \cdot \|_{X}\). We denote by \(L^2(0,T;X)\) the space of measurable functions \(\nu :(0,T)\to X\) such that
	\[
	\|\nu\|_{L^2(0,T,X)} = \left(\int_0^T\|\nu(\cdot,s)\|_{X}^2 ds\right)^{1/2}< +\infty.
	\]
	
	We define the space \(C^k(0,T;X)\) of functions \(\nu\) from \([0,T]\) into \(X\) belonging to \(C^k([0,T])\), and similarly define \(H^k(0,T;X)\) of functions \(\nu\) from \([0,T]\) into \(X\) belonging to \(H^k([0,T])\), where \(H^k([0,T])\) is the Sobolev space of functions with square-integrable derivatives up to order \(k\). We denote by \((\cdot,\cdot)\) the inner product on \(L^2(\Omega)\) or \(L^2(\Omega)^d\), and, by extension, the duality pairing between \(H^{-1}(\Omega)\) and \(H_0^1(\Omega)\), where \(\Omega \subset \mathbb{R}^d\) is a domain with \(d\in \{1,2,3\}\).
	
	\subsection{Notation relative to anisotropic finite element domains $T$}\label{sec22}
	
	Here, $T$ can be a triangle, a rectangle, a tetrahedron, a (rectangular) hexahedron, or a prismatic pentahedron. Parts of the analysis require reference elements \(\widehat{T}\) that can be obtained from the actual element $T$ via some affine linear transformation. The Table $1$ below lists the reference elements for each case. Furthermore for an element $T$, we define two or three anisotropy vectors \(p_{i,T}\), \(i = 1,2,3\), that reflect the main anisotropy directions of that element. These anisotropy vectors are pairwise orthogonal and described in Table $1$.

	\begin{table}[H]
		\centering
		\caption{Different elements, reference elements, and anisotropy vectors. This table outlines various geometric elements, their corresponding reference elements, and the specific anisotropy vectors used for each type.}
		\begin{tabular}{|c|c|c|}
			\hline
			\textbf{Element T} & \textbf{Reference element $\widehat{T}$} & \textbf{Anisotropy vectors $p_{i,T}$} \\
			\hline
			Triangle & $0 \leq x, y \leq 1$, $0 \leq x + y \leq 1$ & $p_{1,T}$ longest edge \\
			& & $p_{2,T}$ height vector \\
			\hline
			Rectangle & $0 \leq x, y \leq 1$ & $p_{1,T}$ longest edge \\
			& & $p_{2,T}$ height vector \\
			\hline
			Tetrahedron & $0 \leq x, y, z \leq 1$, $0 \leq x + y + z \leq 1$ & $p_{1,T}$ longest edge \\
			& & $p_{2,T}$ height in largest face containing $p_{1,T}$ \\
			& & $p_{3,T}$ remaining height \\
		\hline
			Hexahedron & $0 \leq x, y, z \leq 1$ & $p_{1,T}$ longest edge \\
			& & $p_{2,T}$ height in largest face containing $p_{1,T}$ \\
			& & $p_{3,T}$ remaining height \\
			\hline
			Pentahedron (Prism) & $0 \leq x, y, z \leq 1$, $0 \leq x + y \leq 1$ & $p_{1,T}$ longest edge \\
			& & $p_{2,T}$ height in largest face containing $p_{1,T}$ \\
			& & $p_{3,T}$ remaining height \\
			\hline
		\end{tabular}
		\label{tab:elements}
	\end{table}
	
	In addition, the anisotropy vectors \(p_{i,T}\) are enumerated such that \(|p_{1,T}|\geq |p_{2,T}|\geq |p_{3,T}|\) and arranged columnwise to define a matrix \(C_T\in \mathbb{R}^{2\times 2}\) or \(C_T\in \mathbb{R}^{3\times 3}\). This matrix describes the anisotropic orientations of the element $T$.
	
	Set \(h_{i,T} = |p_{i,T}|\) for any \(i = 1,2,3\) and \(H_T = C_T^{\top}\cdot C_T\). Further in the error analysis, the smallest of such lengths play a crucial role, so let denote it by:
	\[
	h_{\min,T} = \min \{h_{i,T}\}.
	\]
	
	Remark that \(C_T\) is an orthogonal matrix since the vectors \(p_{i,T}\) are pairwise orthogonal. We have:
	\[
	H_T = \operatorname{diag}\{h_{1,T}^2,h_{2,T}^2\} \quad \text{in } 2\mathrm{D}
	\]
	\[
	H_T = \operatorname{diag}\{h_{1,T}^2,h_{2,T}^2,h_{3,T}^2\} \quad \text{in } 3\mathrm{D}.
	\]
	
	Furthermore, introduce the height \(h_{E,T}\) over an edge/face of element \(T\) by:
	\[
	h_{E,T} = \frac{|T|}{|E|}.
	\]
	
	We further need the minimal size \(\rho(E)\) of a rectangular face \(E\in \mathcal{E}_{\square}\), i.e \(\rho(E)\) is the smallest of the lengths of the edges of \(E\).
	
	\section{The continuous, time semi-discrete and full discretization problems}\label{sec3}
	\subsection{Model}
	This paper is focused on the coupled-continuum pipe flow model (proposed by X. Wang \cite{wang2010coupled}) with Dirichlet conditions and defined as follow: given a vector function \(R = (R^m,R^c)\) and \(u_0^m\), find a vector function \(u = (u^m,u^c)\) such as:
	
	\begin{equation}\label{model2}
		\left\{
		\begin{array}{ll}
			S\dfrac{\partial u^m}{\partial t} -\nabla \cdot (K\nabla u^m) = -\dfrac{\alpha_{ex}}{2\pi}(u^m - u^c)\delta_{\Gamma} + R^m & \text{in } [0,T]\times \Omega_m \\[10pt]
			-\dfrac{\partial}{\partial x}\left(D\dfrac{\partial u^c}{\partial x}\right) = \alpha_{ex}\left(\dfrac{1}{2\pi}\int_{\Gamma_x}u^m dl_x - u^c\right) + R^c & \text{in } [0,T]\times \Omega_c \\[10pt]
			u = 0 & \text{on } \partial \Omega_m \times \partial \Omega_c \\[5pt]
			u^m(0,\cdot) = u_0^m & \text{on } \Omega_m
		\end{array}
		\right. \quad 
	\end{equation}
	
	where \(u^m\) and \(u^c\) represent the hydraulic head of the matrix and conduit, \(\Omega_{m}\) and \(\Omega_{c}\) are regions occupied by the matrix and the conduit (conceptualized to a one dimensional curve) respectively, \(\Gamma\) is the boundary of the circular horizontal conduit centered at \(x\)-axis, \(\Gamma_{x}\) is the cross section of \(\Gamma\) at \(x\) (a circle), \(dl_x\) represents the infinitesimal increment of arc length on \(\Gamma_{x}\) (equivalent to \(r(x)d\theta\) in the cylindrical coordinates with \(r(x)\) being the radius and \(\theta\) being the angle), and \(|\Gamma_{x}|\) is the length of \(\Gamma_{x}\) which is \(\pi d(x) = 2\pi r(x)\).
	
	For the sake of simplicity, let's consider \(\Omega_{m} = (a,b)\times B_{yz}\) with \(a\) and \(b\) two distinct real numbers satisfying \(a<b\), \(B_{yz}\) a bounded domain of \(\mathbb{R}^2\) containing \((y_0,z_0)\) a fixed element of \(\mathbb{R}^2\), \(\Omega_{c} = (a,b)\times \{(y_{0},z_{0})\}\) a 3D line embedded in a 1D line; \(\Gamma\) a cylinder centered on \(\Omega_{c}\) and its radius is a fixed real number \(r\); and \(\Gamma_{x}\) is a circle of radius \(r\) with center \((y_0,z_0)\) in the \(y\)-\(z\) plane at position \(x\) (see Fig. \ref{fig:domain}). In comprehension, \(\Gamma\) is defined by:
	\[
	\Gamma = \{(x,y,z)\in \Omega_{m}\;|\;(x,y,z) = s_{0}(x,\theta) \text{ with } x\in (a,b) \text{ and } \theta \in [0,2\pi)\}
	\]
	where \(s_0(x,\theta) = (x,y_0 + r\cos \theta ,z_0 + r\sin \theta)\).
	
	\begin{figure}[H]
		\centering
		\begin{tikzpicture}[scale=1.2]
			\draw[thick] (0,0) -- (5,0) -- (5,3) -- (0,3) -- cycle;
			\draw[thick] (0,0) -- (0.5,-0.5) -- (5.5,-0.5) -- (5,0);
			\draw[thick] (5,0) -- (5.5,-0.5);
			\draw[thick] (5,3) -- (5.5,2.5);
			\draw[thick] (0,3) -- (0.5,2.5) -- (5.5,2.5) -- (5,3);
			\draw[thick] (0.5,2.5) -- (0.5,-0.5);
			
			\node at (2.5,1.5) {$\Omega_m$};
			
			\draw[fill=gray!40, opacity=0.6] (2,0.5) circle [x radius=0.8, y radius=0.3];
			\draw[fill=gray!40, opacity=0.6] (2,2.5) circle [x radius=0.8, y radius=0.3];
			\draw[thick] (1.2,0.5) -- (1.2,2.5);
			\draw[thick] (2.8,0.5) -- (2.8,2.5);
			
			\node at (3.5,1.8) {$\Gamma$};
			
			\draw[white, line width=2pt] (2,0.5) -- (2,2.5);
			\draw[dashed, red, line width=1.5pt] (2,0.5) -- (2,2.5);
			
			\node[red] at (2.3,1.2) {$\Omega_c$};
			
			\draw[->] (0,0) -- (0.5,0) node[right] {$y$};
			\draw[->] (0,0) -- (0,-0.5) node[below] {$z$};
			\draw[->] (0,0) -- (-0.3,0.3) node[left] {$x$};
			
			\draw[<->] (1.2,0.3) -- (1.2,0.7);
			\node at (0.8,0.5) {$r$};
			
			\node at (1.0,2.8) {$\Gamma_x$};
			\end{tikzpicture}
		\caption{Geometric representation of the model domain. This figure illustrates the model domain with different geometric elements, including a pipe conduit, the centering line, and the axis labels. The filled and opaque regions represent various faces of the $3D$ structure, while the gray region \(\Gamma\) indicates the conduit within the domain centered at the dashed red line \(\Omega_{c}\). The regions \(\Omega_{m}\) and \(\Gamma\) are labeled accordingly.}
		\label{fig:domain}
	\end{figure}
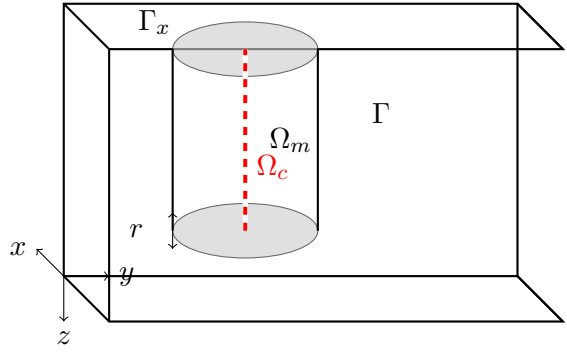
	\subsection{Continuous problem}
	In all that follows, we denote by \(W = H_0^1(\Omega_{m})\times H_0^1(\Omega_{c})\) and \(V = L^2(0,T;W)\). We equip these spaces respectively with the norms \(\| \cdot \|_{W}\) and \(\| \cdot \|_{V}\) given by:
	\begin{equation}
		\begin{array}{rl}
			\| w\|_{W} = \|\nabla w^{m}\|_{\Omega_{m}} + \|\partial_{x}w^{c}\|_{\Omega_{c}},\\[5pt]
			\| v\|_{L^2(\Omega)} = \left(\displaystyle\int_{0}^{T}\| v(t,\cdot)\|_{W}^{2}dt\right)^{1/2}.
		\end{array} \quad (3)
	\end{equation}
	
	We also define on \(L^2(\Omega_m)\times L^2(\Omega_c)\) a norm \(\| \cdot \|_{L^2(\Omega)}\) by:
	\[
	\| v\|_{L^2(\Omega)} = \| w^{m}\|_{L^2(\Omega_m)} + \| w^c\|_{L^2(\Omega_c)}.
	\]
	
	For any \(t \in [0, T]\), we consider on \(W \times W^{m}\) and \(W \times W^{c}\) respectively, the following bilinear forms:
	\[
	a^{m}(u,v^{m}) = K\int_{\Omega_{m}}\nabla u^{m}(x)\cdot \nabla v^{m}(x)dx + a_{ex}\int_{a}^{b}\left[\Pi (u^{m}v^{m})(x) - u^{c}(x)\Pi v^{m}(x)\right]dx
	\]
	\[
	a^{c}(u,v^{c}) = D\int_{a}^{b}\frac{\partial u^{c}}{\partial x}\frac{\partial v^{c}}{\partial x} - a_{ex}\int_{a}^{b}\left(\Pi u^{m}(x) - u^{c}(x)\right)v^{c}(x)dx
	\]
	where \(\Pi : H_0^1(\Omega_m) \to L^2(a,b)\) is the averaging operator given by:
	\[
	\Pi u(x) = \frac{1}{2\pi}\int_{0}^{2\pi}u(s_0(x,\theta))d\theta.
	\]
	
	Obviously, \(\Pi\) is linear and bounded, i.e. there exists \(C_{\Pi} > 0\) such that:
	\[
	\| \Pi v\|_{L^2(a,b)}\leq C_\Pi \| v\|_{H_0^1(\Omega_m)}
	\]
	(see Appendix A for its proof).
	
	Furthermore, by rearranging, one can remark that:
	\begin{equation}
		\begin{array}{r}
			\displaystyle\int_{a}^{b}\left[\Pi (u^{m}v^{m})(x) - u^{c}(x)\Pi v^{m}(x)\right]dx - \int_{a}^{b}\left(\Pi u^{m}(x) - u^{c}(x)\right)v^{c}(x)dx\\[10pt]
			= \dfrac{1}{2\pi}\displaystyle\int_{a}^{b}\int_{0}^{2\pi}(u^{m}(s_{0}(x,\theta) - u^{c}(x))(v^{m}(s_{0}(x,\theta) - v^{c}(x))d\theta dx
		\end{array} 
	\end{equation}
	
	We assume that the datum \(R\) satisfies \(R^{s} \in L^{2}(0,T,H^{-1}(\Omega_{s}))\) and the initial value \(u_{0}^{s} \in L^{2}(\Omega_{s})\) for \(s \in \{m,c\}\). Under these assumptions, one has this result.
	
	\begin{theorem}\label{thm1}
		The weak formulation of (\ref{model2}) is: find a function \(u\) in \(V\) such that:
		\begin{equation}\label{model3}
			\left\{
			\begin{array}{ll}
				S(\partial_t u^m(t,\cdot),w^m) + a^m(u(t,\cdot),w^m) + a^c(u(t,\cdot),w^c) = (R(t,\cdot),w), & \forall w\in W,\forall t\in [0,T]\\[5pt]
				u(0,\cdot) = u_0 & \text{in } \Omega_m \times \Omega_c,
			\end{array}
			\right. \quad 
		\end{equation}
		with \(R\) the linear form: \((R(t,\cdot),w) = (R^{m}(t,\cdot),w^{m})_{\Omega_{m}} + (R^{c}(t,\cdot),w^{c})_{\Omega_{c}}\).
		
		There is a unique solution of this problem.
	\end{theorem}
	
	The following lemma guarantees uniqueness of solution: if \(u_{1}\) and \(u_{2}\) are two solutions of the problem (\ref{model3}), \(u_{1} - u_{2}\) is a solution of (\ref{model3}) with source term \(R = (0,0)\) and initial condition \(u_{0} = 0\); hence applying the following lemma to \(u_{1} - u_{2}\), one can conclude that \(u_{1} - u_{2} = 0\). Here is the lemma.
	
	\begin{lemma}
		If \(u\) is a solution of (\ref{model2}), for any \(t\in [0,T]\), we have the following estimation:
		\begin{equation}\label{model4}
			\begin{array}{l}
				S\| u^{m}(t,\cdot)\| ^2 +\displaystyle\int_0^t\left(K\| \nabla u^m(s,\cdot)\| ^2 +D\left\| \frac{du^c(s,\cdot)}{dx}\right\| ^2\right)\\[10pt]
				\lesssim \displaystyle\int_0^t\left(\frac{1}{K}\| R^m(s,\cdot)\| ^2 + \frac{1}{D}\| R^c(s,\cdot)\| ^2\right)dt + S\| u_0^m\| ^2
			\end{array} \quad 
		\end{equation}
	\end{lemma}
	
	\begin{proof}
		Taking \((u(t,\cdot),u^c)\) as \(w\) in (\ref{model3}), using Cauchy inequality and the fact that two real numbers satisfy \(ab\leq \frac{a^2+b^2}{2}\), we get:
		\[
		\begin{array}{l}
			\displaystyle\int_{\Omega_m}S\frac{\partial u^m(t,\cdot)}{\partial t} u^m(t,\cdot)dx + \int_{\Omega_m}K\nabla u^m(t,\cdot)\cdot \nabla u^m(t,\cdot)dx + D\left\| \frac{du^c(t,\cdot)}{dx}\right\| ^2\\[10pt]
			+\dfrac{a_{ex}}{2\pi}\displaystyle\int_a^b\int_0^{2\pi}\left(u^m(t,s_0(x,\theta)) - u^c(x)\right)^2 d\theta dx \leq \frac{1}{2K}\| R^m(t,\cdot)\| ^2 +\frac{K}{2}\| u^m(t,\cdot)\| ^2\\[10pt]
			+\dfrac{1}{2D}\| R^c(t,\cdot)\| ^2 +\dfrac{D}{2}\| u^c\| ^2.
		\end{array} \quad \label{model7}
		\]
		
		By using Poincaré inequalities on the right hand side, this leads to the following relation:
		\[
		\begin{array}{rl}
			& \dfrac{S}{2}\dfrac{d}{dt}\| u^{m}(t,\cdot)\|^{2} + \dfrac{K}{2}\| \nabla u^{m}(t,\cdot)\|^{2} + \dfrac{D}{2}\left\| \dfrac{d u^{c}}{dx}\right\|^{2} \\[10pt]
			& + \dfrac{a_{ex}}{2\pi}\displaystyle\int_{a}^{b}\int_{0}^{2\pi}\left(u^{m}(t,s_{0}(x,\theta)) - u^{c}(x)\right)^{2}d\theta dx \\[10pt]
			& \lesssim \dfrac{1}{2K}\| R^{m}(t,\cdot)\|^{2} + \dfrac{1}{2D}\| R^{c}(t,\cdot)\|^{2}.
		\end{array} \quad \label{8}
		\]
		
		Integrate on \([0,t]\), for \(t\in [0,T]\), we obtain:
		\[
		\begin{array}{l}
			\dfrac{S}{2}\left(\| u^{m}(t,\cdot)\|^{2} - \| u^{m}(0,\cdot)\| \right) + \displaystyle\int_{0}^{t}\left(K\| \nabla u^{m}(s,x)\|^{2} + D\left\| \frac{\partial u^{c}(s,x)}{\partial x}\right\|^{2}\right)ds\\[10pt]
			+\dfrac{a_{ex}}{2\pi}\displaystyle\int_{0}^{t}\int_{a}^{b}\int_{0}^{2\pi}\left(u^{m}(s,s_{0}(x,\theta)) - u^{c}(x)\right)^{2}d\theta dx ds\\[10pt]
			\lesssim \displaystyle\int_{0}^{t}\left(\frac{1}{2K}\| R^{m}(s,\cdot)\|^{2} + \frac{1}{2D}\| R^{c}(s,\cdot)\|^{2}\right)ds
		\end{array} \quad \label{9}
		\]
		
		Since the quantity \(\frac{a_{ex}}{2\pi}\displaystyle\int_{a}^{b}\displaystyle\int_{0}^{2\pi}\left(u^{m}(t,s_{0}(x,\theta)) - u^{c}(x)\right)^{2}d\theta dx\) is positive, inequality (\ref{model4}) holds.
	\end{proof}
	
	The existence of solution for the problem (\ref{model3}) will be proved with a Galerkin method and the Cauchy-Lipschitz theorem.
	
	
	\begin{proof} (Theorem \ref{thm1})
		We prove the existence and uniqueness of a weak solution using the Galerkin approximation method, which consists in constructing a sequence of finite-dimensional approximate problems and then passing to the limit.
		
		\textbf{step.1: Construction of finite-dimensional approximation spaces:}
		Since \(W = H_0^1(\Omega_m)\times H_0^1(\Omega_c)\) is a separable Hilbert space, it has a Hilbertian basis denoted by \((e_n)_{n\in \mathbb{N}}\) with \(e_n = (e_n^m,e_n^c)\). Let \((W_n)_{n\in \mathbb{N}}\) be a sequence of finite-dimensional subspaces of \(W\) defined by: \(W_n = \operatorname{span}(e_0,e_1,\ldots ,e_n)\). Since \((e_n)_{n\in \mathbb{N}}\) is a Hilbertian basis of \(W\) and \(u_0\in H_0^1(\Omega)\), there is a sequence \(u_{0n}^{m}\) in \(W\) such that:
		\[
		\forall n\in \mathbb{N}, u_{0n}\in W_n, \quad \| u_{0n}^m -u_0^m\|_{H_0^1(\Omega_m)} \xrightarrow[n\to\infty]{} 0.
		\]
		
		\textbf{step.2: Formulation of finite-dimensional problems:}
		Consider the problem: find \(u_n\) in \(L^2(0,T;pr_1(W_n))\times pr_2(W_n)\) such that:
		\[
		\left\{
		\begin{array}{l}
			S(\partial_t u_n^m(t,\cdot),w_n^m) + a^m(u(t,\cdot),w_n^m) + a^c(u_n(t,\cdot),w_n^c) = (R(t,\cdot),w_n),\\[5pt]
			\forall w_n\in W_n
		\end{array}
		\right. \quad \label{model8}
		\]
		where \(pr_1:W_n\longrightarrow H_0^1(\Omega_m)\) and \(pr_2:W_n\longrightarrow H_0^1(\Omega_c)\) are canonical projections.
		
		Since \(W_n\) is a finite-dimensional space, the problem (\ref{model8}) is straightforwardly equivalent to:
		\[
		\left\{
		\begin{array}{l}
			S(\partial_t u_n^m(t,\cdot),e_k^m) + a^m(u(t,\cdot),e_k^m) + a^c(u_n(t,\cdot),e_k^c) = (R(t,\cdot),e_k),\\[5pt]
			\forall k\in \{0,1,\ldots ,n\}.
		\end{array}
		\right. \quad \label{model9}
		\]
		
		We seek a solution \(u_n\) in the finite-dimensional space \(W_n\) and we can decompose this solution with the basis vectors. Hence, we set: \(u_n(t,x) = \sum_{k=1}^n \alpha_k^n(t) e_k(x)\), where \(\alpha_k^n : t\in [0,T] \longrightarrow \alpha_k^n(t) \in \mathbb{R}\) are scalar functions. We just need to find the functions \(\alpha_k^n\) to get solution \(u_n\) of the problem (\ref{model8}). And such functions satisfy:
		\[
		\frac{d\alpha_k^n(t)}{dt}\int_{\Omega_m} Se_k^m e_l^m + \alpha_k^n(t)\int_{\Omega_m} K\nabla e_k^m(x)\nabla e_l^m(x)dx + \alpha_k^n(t)\int_a^b D\frac{de_k^c(x)}{dx}\frac{de_l^c(x)}{dx}dx
		\]
		\[
		+ \alpha_k^n(t)\frac{\alpha_{ex}}{2\pi}\int_a^b\int_0^{2\pi} \left(e_k^m(s_0(x,\theta))-e_k^c(x)\right)\left(e_l^m(s_0(x,\theta))-e_l^c(x)\right)d\theta dx
		\]
		\[
		= (R^m(t,\cdot),e_l^m)+(R^c(t,\cdot),e_l^c), \quad \forall k,l\in\{0,n\}.
		\]
		
		This equation can take the following form: \(a_{kl}\frac{d\alpha_k^n(t)}{dt} + b_{kl}\alpha_k^n(t) = c_k(t)\) for any \(k,l\in\{0,n\}\). Note that \(a_{kl}\), \(b_{kl}\), and \(c_k(t)\) are obviously defined. Hence, the problem (\ref{model8}) is equivalent to the ordinary differential equation: \(A\frac{d\alpha^n(t)}{dt} + B\alpha^n(t) = C(t)\) with \(A = (a_{kl})_{0\leq k,l\leq n}\), \(B = (b_{kl})_{0\leq k,l\leq n}\), \(C(t) = (c_k(t))_{0\leq k\leq n}\) and \(\alpha^n(t) = (\alpha_k^n(t))_{0\leq k\leq n}\).
		
		Since \(A\) is a positive definite matrix, this problem has a unique solution thanks to the Cauchy-Lipschitz theorem. We can conclude that the problem (\ref{8}) has a unique solution \(u_n\).
		
		\textbf{step. 3: Convergence of the sequence \((u_n)_n\):}
		With the same manipulations, the estimation (\ref{model4}) holds for problem (\ref{model8}). Hence, we have:
		\begin{eqnarray}\label{model12}\nonumber
		S\|u_n^m(t,\cdot)\|^2 + \int_0^t\left(K\|\nabla u_n^m(s,\cdot)\|^2 ds + D\left\|\frac{du_n^c(s,\cdot)}{dx}\right\|^2\right)
		&\lesssim& \int_0^t\left(\frac{1}{K}\|R^m(s,\cdot)\|^2 + \frac{t}{D}\|R^c(s,\cdot)\|^2\right)\\
		 &+& S\|u_{0n}^m\|^2.
		\end{eqnarray}
		
		This means that the sequence \((u_n)_n\) is bounded in the space \(V = L^2(0,T;W)\). Hence, thanks to Banach-Alaoglu theorem \cite{brezis1983analyse}, there is a subsequence \((u_{n_k})_k\) and a function \(u\in V\) that satisfies for any \(w\in W\):
		\[
		\int_{\Omega_m}\frac{\partial u_{n_k}^m(t,x)}{\partial t} w^m(x)dx \xrightarrow[k\to\infty]{} \int_{\Omega_m}\frac{\partial u^m(t,x)}{\partial t} w^m(x)dx
		\]
		\[
		\int_{\Omega_m}\nabla u_{n_k}^m(t,x)\cdot\nabla w^m(x)dx \xrightarrow[k\to\infty]{} \int_{\Omega_m}\nabla u^m(t,x)\cdot\nabla w^m(x)dx
		\]
		\[
		\int_{\Omega_m}(u_{n_k}^m(x,r,\theta)-u_{n_k}^c(x))w^m \xrightarrow[k\to\infty]{} \int_{\Omega_m}(u^m(x,r,\theta)-u^c(x))w^m
		\]
		\[
		\int_{\Omega_m}(u_{n_k}^m(x,r,\theta)-u_{n_k}^c(x))w^c \xrightarrow[k\to\infty]{} \int_{\Omega_m}(u^m(x,r,\theta)-u^c(x))w^c
		\]
		\[
		\int_a^b\frac{du_{n_k}^c}{dx}\frac{dw^c}{dx} \xrightarrow[k\to\infty]{} \int_a^b\frac{du^c}{dx}\frac{dw^c}{dx}.
		\]
		
		Therefore, the limit \(u\) satisfies the following:
		\[
		S(\partial_t u^m(t,\cdot),w^m) + a^m(u(t,\cdot),w^m) + a^c(u(t,\cdot),w^c) = (R(t,\cdot),w), \quad \forall w\in W, \forall t\in[0,T] \quad \label{model10}
		\]
		
		\textbf{step. 4: Verification of initial condition:} \(u(0,\cdot) = u_0\).
		
		First, we observe that from the estimation (\ref{model4}) applied to the approximate solutions \(u_{n_k}\), we have uniform bounds on:
		\[
		\left\|\frac{\partial u_{n_k}^m}{\partial t}\right\|_{L^2(0,T;H^{-1}(\Omega_m))}, \quad \|\nabla u_{n_k}^m\|_{L^2(0,T;L^2(\Omega_m))}, \quad \left\|\frac{du_{n_k}^c}{dx}\right\|_{L^2(0,T;L^2(\Omega_c))} \quad\label{model11}
		\]
		
		From these bounds and the fact that \(u_{n_k} \rightharpoonup u\) weakly in \(V\), we can deduce that:
		\[
		u_{n_k}^m \rightharpoonup u^m \text{ weakly in } L^2(0,T;H_0^1(\Omega_m)) \quad \label{model14}
		\]
		\[
		\frac{\partial u_{n_k}^m}{\partial t} \rightharpoonup \frac{\partial u^m}{\partial t} \text{ weakly in } L^2(0,T;H^{-1}(\Omega_m)) \quad \label{model15}
		\]
		
		Applying the Lions-Magenes lemma, we conclude that \(u^m \in C(0,T;L^2(\Omega))\). Thus, we can meaningfully evaluate \(u^m(0)\).
		
		Recall that for each \(k\), we have: \(u_{n_k}^m(0,\cdot) = u_{0n_k}^m\), and by construction \(u_{0n_k}^m \to u_0^m\) in \(L^2(\Omega_m)\).
		
		For any function \(\varphi \in C^1([0,T])\) with \(\varphi(T) = 0\), and any \(w \in W\), we have:
		\[
		\int_0^T \left(\frac{\partial u_{n_k}^m}{\partial t}, w^m\right) \varphi(t) dt = \int_0^T -(u_{n_k}^m, w^m)\varphi'(t) dt + (u_{0n_k}^m, w^m)\varphi(0)
		\]
		by integration by parts. Taking the limit as \(k \to \infty\):
		\[
		\int_0^T \left(\frac{\partial u^m}{\partial t}, w^m\right) \varphi(t) dt = \int_0^T -(u^m, w^m)\varphi'(t) dt + (u_0, w^m)\varphi(0) \quad \label{model16}
		\]
		
		But, on the other hand, by definition, we have:
		\[
		\int_0^T \left(\frac{\partial u^m}{\partial t}, w^m\right) \varphi(t) dt = \int_0^T -(u^m, w^m)\varphi'(t) dt + (u^m(0), w^m)\varphi(0) \quad (\label{model17})
		\]
		
		Comparing these equations and noting that \(\varphi(0)\) can be arbitrary, we conclude:
		\[
		(u_0^m, w^m) = (u^m(0), w^m) \quad \text{for all } w^m \in H_0^1(\Omega_m) \quad \label{model18}
		\]
		
		Therefore, \(u^m(0) = u_0^m\). This completes the proof of Theorem \ref{thm1}, establishing both the existence and the correct initial condition for the solution.
	\end{proof}
	
	\subsection{Time discretization}
	
	Let \((t_p)_{0\leq p\leq N}\) be a discretization of \([0,T]\) that satisfies: \(0 = t_0 < t_1 < \dots < t_{N-1} < t_N = T\). Set \(\tau_p = t_p - t_{p-1}\) for any \(p\in \{1,2,\ldots ,N\}\).
	
	Using backward Euler scheme, the semi-discrete approximation of the continuous problem (\ref{model3}) consists in finding the sequence \((u_p)_{1\leq p\leq N}\) of \(W\) such that: for any \(p\in \{1,2,\ldots ,N\}\)
	\begin{equation}\label{19}
		a_p(u_p,w) = S(u_{p-1}^m,w^m) + \tau_p(R_p,w), \quad \forall w\in W,
	\end{equation}
	with \(R_p = R(t_p,\cdot)\), \(a_p(v,w) = a_p^m(v,w^m) + \tau_p a^c(v,w^c)\) and \(a_p^m(v,w^m) = S(v^m,w^m) + \tau_p a^m(v,w^m)\).
	
	For a fixed \(p\) and for the norm \(\| \cdot \|_{W}\), the bilinear form \(a_p\) is continuous on \(W\times W\) and coercive on \(W\) and the linear form \(w \mapsto S(u_{p-1}^m,w^m) + \tau_p(R_p,w)\) is continuous on \(W\); thus the Lax-Milgram theorem ensures that there exists a unique solution to the problem (\ref{19}).
	
	\subsection{Full discretization}\label{fulldis}
	
	Now the problem (\ref{19}) is discretized by some nonconforming finite element method on anisotropic mesh in the porous matrix \(\Omega_m\). For each \(p\in [1,N]\), let \((\mathcal{T}_{ph})_{h>0}\) be a family of conforming triangulation (in the standard sense of \cite{ciarlet2002finite}) of \(\Omega_m\) based on anisotropic tetrahedra, anisotropic rectangular hexahedra, or anisotropic rectangular pentahedra with mesh size \(h > 0\); and \(\mathcal{P}_{ph}\) a family of triangulation of \(\Omega_c\) based on segments that have nodes belonging to the mesh \(\mathcal{T}_{ph}\).
	
	We assume that the anisotropic nature of each element in \(\mathcal{T}_{ph}\) satisfies the conditions described in Sections 3.3 and 3.4 of \cite{creuse2004posteriori}.
	
	Each element of the triangulation \(\mathcal{T}_{ph}\) is denoted by \(T\) or \(T_i\). The set of faces of an element \(T\) is denoted by \(\mathcal{F}_T\) and the set of edges of an element/face \(T\) is denoted by \(\mathcal{E}_T\).
	
	We denote by \(\mathcal{F}_{ph}\) the set of all faces in \(\mathcal{T}_{ph}\); by \(\mathcal{F}_{ph}^{\mathrm{int}}\) the set of all interior faces; by \(\mathcal{F}_{ph}^{\square}\) the set of all rectangular faces; by \(\mathcal{E}\) the set of all edges; by \(\omega_F\) the set of all elements sharing an edge/face \(F\); and by \(\omega_x\) the set of all elements sharing a node \(x\) of the mesh \(\mathcal{T}_{ph}\) or \(\mathcal{P}_{ph}\).
	
	We approach the space \(W\) by a functional space \(W_{ph} = W_{ph}^m \times W_{ph}^c\), with \(W_{ph}^m\) and \(W_{ph}^c\) spaces of continuous, piecewise linear functions over \(\mathcal{T}_{ph}\) and \(\mathcal{P}_{ph}\) respectively. The space \(W_{ph}\) will be defined for some variants of the Crouzeix-Raviart finite element methods in Subsection 3.3. For each method, we will provide the appropriate approximation space \(W_{ph}\), and equip it with the following norm:
	\[
	\| w_{ph}\|_{W_{ph}^m} = \| w^m\|_{W_{ph}^m} + \| w^c\|_{W_{ph}^c}
	\]
	\[
	\| w^m\|_{W_{ph}^m} = \sum_{T\in \mathcal{T}_{ph}} \| \nabla w_{ph}^m\|_{T}
	\]
	\[
	\| w^c\|_{W_{ph}^c} = \sum_{E\in \mathcal{P}_{ph}} \| \nabla w_{ph}^c\|_{E}.
	\]
	
	Here, \(\nabla w_{ph}^m\) and \(\nabla w_{ph}^c\) denote broken gradients, computed element-wise on \(\mathcal{T}_{ph}\) and \(\mathcal{P}_{ph}\), respectively. This allows for discontinuities of the gradient across element interfaces.
	
	We consider the following bilinear form \(a_{ph}\) defined on \(W_{ph}\times W_{ph}\) by: \(a_{ph}(v_{ph},w_{ph}) = a_{ph}^m(v_{ph},w_{ph}) + \tau_p a^c(v_{ph},w_{ph}^c)\) where:
	\[
	a_{ph}^m(v_{ph},w_{ph}^m) = \sum_{T\in \mathcal{T}_{ph}} \left[ S\int_T v_{ph}^m w_{ph}^m + \tau_p K\int_T \nabla v_{ph}^m \cdot \nabla w_{ph}^m \right] + \tau_p \sum_{E\in \mathcal{P}_{ph}} \alpha_{ex}\int_E \left[ \Pi(v_{ph}^m w_{ph}^m) - v_{ph}^c \Pi w_{ph}^m \right]
	\]
	
	The full discretization of the problem (\ref{model2}) has a unique weak solution satisfying: given an approximation \(u_{0h}^m\) of a given function \(u_0^m \in H_0^1(\Omega_m)\), find a sequence \((u_{ph})_{1\leq p\leq N}\) in \(W_{ph}\) such that:
	\begin{equation}\label{20'}
		a_{ph}(u_{ph},w_{ph}) = \tau_p (R_p,w_{ph}) + S(u_{p-1,h}^m, w_{ph}^m)_{\Omega_m}, \quad \forall w_{ph} \in W_{ph}.
	\end{equation}
	
	For any \(p\) and \(h > 0\), the bilinear form \(a_{ph}\) is continuous on \(W_{ph} \times W_{ph}\) and coercive in \(W_{ph}\), and the linear form \(w_{ph} \mapsto S(u_{p-1,h}^m, w_{ph}^m) + \tau_p(R_p,w_{ph})\) is continuous in \(W_{ph}\); thus the Lax-Milgram theorem guarantees that the problem (\ref{20'}) has a unique solution. In indeed, we have the following theorem:
	\begin{theorem}
	Let $p\in \{1,\ldots,N\}$ be fixed. \\
	Assume that 
	$$S> 0, \mbox{  } K>0, \mbox{   } D> 0,\mbox{  } \alpha_{ex}> 0,$$
	and let $$W_{ph}=W_{ph}^m\times W_{ph}^c$$ be the finite-dimensional discrete space defined in Section \ref{fulldis}.\\ Suppose that the bilinear form $a_{ph}$ satisfies 
	$$a_{ph}(v_{ph},v_{ph})\geq C_0 \parallel v_{ph}\parallel_{W_{ph}}^2, \quad \forall v_{ph}\in W_{ph},$$
	for some constant $C_0> 0$ independent of $v_{ph}$.\\
	Then the fully discrete problem
	\begin{eqnarray}
		a_{ph}(u_{ph}^p,w_{ph}^p)=\tau_p(R^p,w_{ph})+S(u_{p-1,h}^{m},w_{ph}^m)_{\Omega_m}, \quad \forall w_{ph}\in W_{ph}
	\end{eqnarray}
	admits a unique solution $$u_{ph}^p\in W_{ph}.$$
	\end{theorem}
	\begin{proof}
	\mbox{ }\\
	\textbf{Step 1. Definition of the discrete bilinear form.}\\  For $v_{ph}, w_{ph}\in W_{ph}$, define:
	$$
	a_{ph}(v_{ph},w_{ph})=a_{ph}^m(v_{ph},w_{ph}^m)+\tau_p a_{ph}^p(v_{ph},w_{ph}^c),
	$$
	where 
	\begin{eqnarray}
	a_{ph}^m(v_{ph},w_{ph}^m)&=&\displaystyle\sum_{T\in\mathcal{T}_{ph}}\left[
	S\int_T v_{ph}^m w_{ph}^m +\tau_p  K\int_T \nabla_h v_{ph}^m\cdot\nabla_h w_{ph}^m
	\right]\\
	&+&\tau_p\alpha_{ex}\displaystyle\sum_{E\in\mathcal{P}_{ph}}\int_E\left(
	\Pi(v_{ph}^m w_{ph}^m)-v_{ph}^c\Pi w_{ph}^m
	\right),
	\end{eqnarray}
	and
	 $$
	a_{ph}^c(v_{ph},w_{ph}^c)=\displaystyle\sum_{E\in\mathcal{P}_{ph}}\int_E
	\left(D\nabla_h v_{ph}^c\cdot \nabla_h w_{ph}^c+\alpha_{ex} v_{ph}^c w_{ph}^c\right)
	$$

	Since all integrals are finite sums over finite-dimensional polynomial, $a_{ph}$ is well defined.\\
	\textbf{Step 2. Continuity.}\\
	Let $v_{ph}, w_{ph\in W_{ph}}$. \\
	Using the Cauchy-Schwarz inequality, 
	$$\left|\displaystyle\sum_{T\in\mathcal{T}_{ph}} \int_T \nabla_h v_{ph}^m\cdot \nabla_h w_{ph}^m\right|
	\leq \parallel \nabla_h v_{ph}^m\parallel_{L^2(\Omega_m)}\parallel \nabla_h w_{ph}^m\parallel_{L^2(\Omega_m)}.
		$$
		Similarly,
		$$
		\left|\displaystyle\sum_{T\in\mathcal{T}_{ph}}\int_T v_{ph}^m w_{ph}^m\right|\leq \parallel v_{ph}^m\parallel_{L^2(\Omega_m)}\parallel w_{ph}^m\parallel_{L^2(\Omega_m)}.
		$$
		Since the averaging operator $$\Pi:H^1(\Omega_m)\rightarrow L^2(\Omega_c)$$
		is bounded,
		$$\parallel \Pi v\parallel_{L^2(\Omega_c)}\leq C_{\Pi}\parallel v\parallel_{H^1(\Omega_m)},$$
		the coupling terms satisfy:
		$$\left|\int_{E} v_{ph}^c \Pi w_{ph}^m\right|\leq C_{\Pi}\parallel v_{ph}^c\parallel_{L^2(E)}
		\parallel w_{ph}^m\parallel_{H^1(\Omega_m)}.
		$$
		Combining all estimates gives,
		$$|a_{ph(v_{ph},w_{ph})}|\leq C \parallel v_{ph}\parallel_{W_{ph}}\parallel w\parallel_{W_{ph}},$$
		for some constant $C> 0$.\\
		Hence $a_{ph}$ is continuous on $W_{ph}\times W_{ph}$.\\\\
		\textbf{Step 3. Coercivity. }\\
		Taking $w_{ph}=v_{ph}$ yields:
		\begin{eqnarray*}
			a_{ph}(v_{ph},v_{ph})&=&S\parallel v_{ph}^m\parallel_{L^2(\Omega)}^2+\tau_p K
			\parallel\nabla_h v_{ph}^m\parallel_{L^2(\Omega)}^2\\
			&+&
			\tau_p D\parallel \nabla_h v_{ph}^c\parallel_{L^2(\Omega_c)}^2+\tau_{p}\alpha_{ex}\parallel v_{ph}^c\parallel_{L^2(\Omega_c)}\\
			&+&
			\frac{\tau_p\alpha_{ex}}{2\pi}\int_{a}^b\int_{0}^{2\pi}\left(v_{ph}^m(s_0(x,\theta))-v_{ph}^c\right)^2 d\theta dx.
		\end{eqnarray*}
	The last term is nonnegative.\\
	Therefore, 
	\begin{eqnarray*}
	a_{ph}(v_{ph},v_{ph})&\geq& S \parallel v_{ph}^m\parallel_{L^2(\Omega_m)}^2+\tau_p K
	\parallel \nabla_h v_{ph}^m\parallel_{L^2(\Omega_m)}^2\\
	&+& \tau_p D\parallel \nabla_h v_{ph}^c\parallel_{L^2(\Omega_c)}^2.
	\end{eqnarray*}
Using the discrete Poincaré inequality for nonconforming finite element space, 
$$\parallel v_{ph}^m\parallel_{L^2(\Omega_m)}\leq C_p \parallel \nabla_h v_{ph}^m\parallel_{L^2(\Omega_m)},$$
and  the standard Poincaré inequality on $\Omega_c$,
$$\parallel v_{ph}^c\parallel_{L^2(\Omega_c)}\leq C_p^c\parallel \nabla_h v_{ph}^m\parallel_{L^2(\Omega_c)},$$
we obtain, 
$$a_{ph}(v_{ph},v_{ph})\geq C_0\left(\parallel \nabla_h v_{ph}^m\parallel_{L^2(\Omega_m)}^2+
\parallel \nabla_h v_{ph}^c\parallel_{L^2(\Omega_c)}^2
\right)$$
for some constant $C_0> 0$. Hence,
$$a_{ph}(v_{ph},v_{ph})\geq C_0\parallel v_{ph}\parallel_{W_{ph}}^2.$$
Thus, $a_{ph}$ is coercive. 
Continuity of the right-hand side is trivially. \\
The proof is complete. 

	\end{proof}
	\subsection{Examples of finite elements for the space \(W_{ph}^m\)}
	
	Defining a different finite element does not matter for the a-posteriori error analysis to be performed later. The difference occurs in its implementation.
	
	As in the standard theory, a finite element is denoted by a triplet \((T,P,\Sigma)\) where \(T\) is a domain, \(P\) denotes a space of functions, and \(\Sigma\) is a set of functionals of \(P^*\) (space of linear forms defined on \(P\)) \cite{ciarlet2002finite}.
	
	\subsubsection{Crouzeix-Raviart elements I}
	
	For a triangulation \(\mathcal{T}_{ph}\) of \(\Omega_m\) consisting of tetrahedra, we approximate the hydraulic head in porous matrix \(\Omega_m\) in the Crouzeix-Raviart element space and the hydraulic head in pipe conduit in the space of piecewise functions, namely:
	\begin{eqnarray}\label{21}
		\begin{array}{l}
			W_{ph}^m = \{ u_{ph}^m \in L^2(\Omega_m) : u_{ph}^m|_T \in P_1(T), \forall T \in \mathcal{T}_{ph}, \int_E [u_{ph}^m] = 0, \forall E \in \mathcal{E}\},\\\quad
			W_{ph}^c = \{ u_{ph}^c \in C^0(\Omega_c) : u_{ph}^c|_E \in P_1(E), \forall E \in \mathcal{P}_{ph} \text{ and } u_{ph}^c|_{\Omega_c} = 0 \}.
		\end{array}
	\end{eqnarray}
	
	\subsubsection{Modified Crouzeix-Raviart elements on pentahedra}
	
	Consider a triangulation \(\mathcal{T}_{ph}\) of porous matrix \(\Omega_m\) made of rectangular pentahedra. On the reference pentahedron \(\widehat{T}\), one defines a finite element \((\widehat{T},\widehat{P},\widehat{\Sigma})\) by setting: \(\widehat{P} = P_2\) and \(\widehat{\Sigma} = \{l_i\}_{1\leq i\leq 10}\) where:
	\[
	l_i(p) = \int_{\widehat{E}_i} p, \text{ for } i = 1,\dots,5,\quad l_6(p) = \int_{\widehat{T}} p,\quad l_{i+6}(p) = \int_{\widehat{T}} p \tilde{q}_i, \text{ for } i = 1,\dots,4,
	\]
	with \(\widehat{E}_i\) a face of pentahedron \(\widehat{T}\) and
	\[
	\tilde{q}_1(x,y,z) = 1-3x-2z+6xz,\quad \tilde{q}_2(x,y,z) = 1-3y-2z+6yz,
	\]
	\[
	\tilde{q}_3(x,y,z) = 1-4x-2y+6xy+3x^2,\quad \tilde{q}_4(x,y,z) = 2x-2y-3x^2+3y^2.
	\]
	
	Note that \(l_i(\tilde{q}_j) = 0\) for \(i = 1,\dots,6\) and \(j = 1,\dots,4\).
	
	Enumerate the faces \(\widehat{E}_i\) of \(\widehat{T}\) such that \(\widehat{E}_1, \widehat{E}_2, \widehat{E}_3, \widehat{E}_4, \widehat{E}_5\) are contained in the planes \(z = 0\), \(z = 1\), \(y = 0\), \(x = 0\) and \(x + y = 1\) respectively. There exists an associated basis \(\{q_i\}_{1\leq i\leq 10}\); here we need to specify the first six functions:
	\[
	q_1(x,y,z) = 2-8z+6z^2,\quad q_4(x,y,z) = 1-6x+6z,
	\]
	\[
	q_2(x,y,z) = 2(-2z+3z^2),\quad q_5(x,y,z) = 1-6(x-x^2+y-y^2-2xy)/2,
	\]
	\[
	q_3(x,y,z) = 1-6y+6y^2,\quad q_6(x,y,z) = 24(x-x^2+y-y^2-xy)+12(z-z^2)-6.
	\]
	
	The main interest is that \(q_1, q_2\) do not depend on \(x, y\) and conversely \(q_3, q_4, q_5, q_6\) do not depend on \(z\).
	
	Consider now the actual (anisotropic) pentahedron \(T\) which can be obtained from the reference pentahedron by an affine transformation. In this way, the finite element \((T,P,\Sigma)\) is also defined, i.e. one has \(q_i(x,y,z) = \hat{q}_i(\hat{x},\hat{y},\hat{z})\) and \(l_i(q) = \hat{l}_i(\hat{q})\).
	
	At this stage, define the approximation space in porous matrix \(W_{ph}^m\) by:
	\[
	W_{ph}^m = \{ w_{ph}^m \in L^2(\Omega_m) : w_{ph}^m|_T \in P_2, \forall T \in \mathcal{T}_{ph}, \int_E [w_{ph}^m]_E = 0, \forall E \in \mathcal{P}_{ph} \},
	\]
	and the approximation space in pipe \(W_{ph}^c\) again by (\ref{21}).
	
	\subsubsection{Modified Crouzeix-Raviart elements on hexahedra}
	
	Consider the 3D case and a triangulation of \(\Omega\) made of rectangular hexahedra. On the reference rectangular hexahedron \(\widehat{T} = (a,b)^3\), define a finite element \((\widehat{T},\widehat{P},\widehat{\Sigma})\) by setting: \(\widehat{P} = \mathbb{P}^2 + \operatorname{span}\{\hat{x}\hat{y}\hat{z}\}\) and \(\widehat{\Sigma} = \{\hat{I}_i\}_{1\leq i\leq 11}\) defined by:
	\[
	\hat{I}_i(p) := \int_{\hat{E}_i} p, \text{ for } i = 1,\ldots ,6,\quad \hat{I}_7(p) := \int_{\widehat{T}} p,\quad \hat{I}_{7+i}(p) := \int_{\widehat{T}} p \hat{q}_i, \text{ for } i = 1,\ldots ,4
	\]
	where
	\[
	\begin{array}{r}
		\hat{q}_1(\hat{x},\hat{y},\hat{z}) := 1 - 2\hat{x} -2\hat{z} +4\hat{x}\hat{z},\quad \hat{q}_2(\hat{x},\hat{y},\hat{z}) := 1 - 2\hat{y} -2\hat{z} +4\hat{y}\hat{z}\\
		\hat{q}_3(\hat{x},\hat{y},\hat{z}) := 1 - 2\hat{x} -2\hat{y} +4\hat{x}\hat{y},\quad \hat{q}_4(\hat{x},\hat{y},\hat{z}) := 1 - \hat{x} -\hat{y} -\hat{z} +4\hat{x}\hat{y}\hat{z}.
	\end{array} \quad (4)
	\]
	
	The above choice is motivated by the fact that \(\hat{I}_i(\hat{q}_j) = 0\) for \(i = 1,\ldots ,7\) and \(j = 1,\ldots ,4\).
	
	Now denote the faces \(\hat{E}_1,\hat{E}_2,\ldots ,\hat{E}_6\) such that they are included in the planes, \(\hat{z} = 0,\hat{z} = 1,\hat{y} = 0,\hat{y} = 1,\hat{x} = 0,\hat{x} = 1\), respectively. Then there exists an associated basis \(\{\hat{q}_i\}_{1\leq i\leq 11}\) whose first six entries are of particular interest, and given by:
	\[
	\hat{q}_1(\hat{x},\hat{y},\hat{z}) = \hat{h}(\hat{z}),\quad \hat{q}_3(\hat{x},\hat{y},\hat{z}) = \hat{h}(\hat{y}),\quad \hat{q}_5(\hat{x},\hat{y},\hat{z}) = \hat{h}(\hat{x}),
	\]
	\[
	\hat{q}_2(\hat{x},\hat{y},\hat{z}) = \hat{h}(1-\hat{z}),\quad \hat{q}_4(\hat{x},\hat{y},\hat{z}) = \hat{h}(1-\hat{y}),\quad \hat{q}_6(\hat{x},\hat{y},\hat{z}) = \hat{h}(1-\hat{x})
	\]
	where we set \(h(t) = -2t + 3t^2\). Again the distinct feature is that \(\hat{q}_1,\hat{q}_2\) depend only on \(\hat{z}\), that \(\hat{q}_3,\hat{q}_4\) depend only on \(\hat{y}\), and \(\hat{q}_5,\hat{q}_6\) depend only on \(\hat{x}\).
	
	The finite element \((T,P,\Sigma)\) on the actual hexahedron \(T\) is obtained by the usual affine transformation.
	
	To define the pair \((W_{ph}^m,W_{ph}^c)\) set:
	\[
	W_{ph}^m = \{ w_{ph}^m \in L^2(\Omega_m) : w_{ph}^m|_T \in [\mathbb{P}^2 + \operatorname{span}\{\hat{x}\hat{y}\hat{z}\}], \forall T \in \mathcal{T}_{ph}, \int_E [w_{ph}^m] = 0, \forall E \in \mathcal{E} \}
	\]
	and \(W_{ph}^c\) defined by (\ref{21}).
	
	\section{Analytical tools}\label{sec4}
	
	We generalize the approximation spaces \(W_{ph}^m\) by allowing the local polynomial degree to be arbitrary. Specifically, we define:
	\[
	W_{ph}^{m,k} = \{ w_{ph}^m \in L^2(\Omega_m) : w_{ph}^m|_T \in \mathbb{P}^k(T), \forall T \in \mathcal{T}_{ph}, \int_E [w_{ph}^m] = 0, \forall E \in \mathcal{E} \},
	\]
	\(k\) an arbitrary integer.
	
	\subsection{Lagrange Interpolation Operator}
	
	We know that \(\Omega_c \subset \mathbb{R}\) and thanks to Sobolev injection, each element of \(H^1(\Omega_c)\) is a \(C^0(\overline{\Omega_c})\) function defined on \(\Omega_c\). Hence, we can define Lagrange interpolant for any \(v \in H^1(\Omega_c)\).
	
	Denote by \(a_0,a_1,a_2,\ldots ,a_{N_c}\) the nodes of the mesh \(\mathcal{P}_{ph}\), and assume that \(a_0 \leq a_1 \leq \dots \leq a_{N_c}\). For each \(i = 0,\ldots ,N_c-1\), define the sub-interval \(E_i = [a_i,a_{i+1}]\) and let \(h_i = |a_{i+1} - a_i|\) denote its length, and the mesh size of \(\mathcal{P}_{ph}\) by \(h_c = \max_{0\leq i\leq N_c} h_i\).
	
	Let \((\phi_i)_{0\leq i\leq N_c}\) be the Lagrange basis of the finite element space \(W_{ph}^c\), consisting of piecewise polynomial functions satisfying the interpolation condition:
	\[
	\phi_i(a_j) = \delta_{ij}, \quad \forall i,j\in \{0,1,\ldots ,N_c\}.
	\]
	
	\begin{definition}
		The Lagrange interpolation operator \(I_L\) is defined from \(H_0^1(\Omega_c)\) into \(W_{ph}^c\) by:
		\[
		I_L v^c = \sum_{i=0}^{N_c} v^c(a_i)\phi_i.
		\]
	\end{definition}
	
	\begin{theorem}[Lagrange Interpolation Estimate for Conduit Domain]
		Let \(v^c \in H_0^1(\Omega_c)\). Then:
		\begin{equation}
			\sum_{E\in \mathcal{P}_{ph}} h_E^{-2} \| v^c - I_L v^c \|_{L^2(E)}^2 \lesssim \| v^c \|_{H^1(\Omega_c)}^2. \quad \label{23}
		\end{equation}
	\end{theorem}
	
	\subsection{Alignment measure}
	
	When performing error analysis on anisotropic meshes, we aim to proceed in a manner similar to the isotropic case. To this end, we derive interpolation estimates and quickly observe that some results resemble those obtained for isotropic meshes. However, a key difference arises: an additional factor appears in the anisotropic setting. This factor reflects how well the anisotropic mesh aligns with the anisotropy of the function. It plays a crucial role in determining the quality (either improvement or degradation) of subsequent estimates. Following the approach introduced in \cite{kunert2000posteriori}, we define the measure of alignment of high-order as follows.
	
	\begin{definition}[Higher-Order Alignment Measure]
		Let \(\mathcal{T}\) be a family of triangulation of \(\Omega_m\). For \((v,\mathcal{T})\in H^1(\Omega_m)\times \mathcal{T}\), the alignment measure is defined as:
		\begin{equation}
			m_1(v,\mathcal{T}) := \frac{\left(\displaystyle\sum_{T\in\mathcal{T}} h_{\min,T}^{-2} \| \nabla v\cdot C_T\|_{T}^2\right)^{1/2}}{\| \nabla v\|}. \quad \label{24}
		\end{equation}
	\end{definition}
	
	Since the maximum singular value of \(C_T\) is \(h_{\max,T} = h_{1,T}\), for any \(v\in H^1(\Omega_m)\) with \(\nabla v \neq 0\), the alignment measure satisfies:
	\begin{equation}
		1 \leq m_2(v,\mathcal{T}) \leq \max_{T\in \mathcal{T}} \frac{h_{\max,T}}{h_{\min,T}}. \quad \label{25}
	\end{equation}
	
	This implies:
	\begin{enumerate}
		\item For isotropic meshes where \(h_{\max,T} = h_{\min,T}\) for all \(T\): \(m_1(v,\mathcal{T}) = 1\).
		\item For highly anisotropic elements: \(m_1(v,\mathcal{T})\) can grow as the aspect ratio. For instance, for moderate anisotropy \((h_{\max}/h_{\min}\approx 10)\): \(m_1 \leq 10\); and for strong anisotropy \((h_{\max}/h_{\min}\approx 100)\): \(m_1 \leq 100\).
	\end{enumerate}
	
	Consequently, a well-aligned mesh \(\mathcal{T}\) with an anisotropic function results in a small aspect ratio ensuring the smallest of \(m_1(v,\mathcal{T})\). The crude upper bound of \(m_1\) confirms (\ref{25}) that is a natural extension of isotropic meshes.
	
	\subsection{Clément interpolation}
	
	\begin{definition}
		Let \(N_{\Omega_m}\) and \(N_{\partial\Omega}\) be respectively the set of interior nodes and boundary nodes.
		
		For each node \(x_j \in N_{\Omega_m} \cup N_{\partial\Omega}\), we define the local \(L^2\)-projection operator \(P_j^{(k)} : L^2(\omega_{x_j}) \to P_k(\omega_{x_j})\) by requiring orthogonality against all polynomials of degree at most \(k-1\):
		\begin{equation}\label{26}
			\int_{\omega_{x_j}} (v - P_j^{(k)}v) \cdot q \, dx = 0 \quad \forall q \in P_{k-1}(\omega_{x_j}).  
		\end{equation}
		
		For any integer \(k \geq 1\), the interpolation operator of Clément \(I_{\mathrm{Cl}}^0\) is defined from \(H_0^1(\Omega_m)\) into \(W_{ph}^{m,k} \cap H_0^1(\Omega_m)\) by:
		\begin{equation}
			I_{\mathrm{Cl}}^0 v := \sum_{x_j \in N_{\Omega_m} \cup N_{\partial\Omega}} P_j^{(k)}(v)(x_j) \phi_j^{(k)},  \label{27}
		\end{equation}
		where \(\phi_j^{(k)} \in W_{\mathrm{Cl}}^{m,k}\) are the higher-order nodal basis functions.
		
		Furthermore, for any integer \(k \geq 1\), the interpolation of Clément \(I_{\mathrm{Cl}}\) is defined from \(H^1(\Omega_m)\) into \(W_{ph}^{m,k} \cap H^1(\Omega_m)\) by:
		\begin{equation}
			I_{\mathrm{Cl}} v := \sum_{x_j \in N_{\Omega_m} \cup N_{\partial\Omega}} P_j^{(k)}(v)(x_j) \phi_j^{(k)}. \label{28}
		\end{equation}
		
		For vector-valued functions, these operators are defined component-wise.
	\end{definition}
	
	\begin{theorem}[Clément Interpolation Estimates for the Matrix Domain]
		Let \(k \geq 1\) be an integer. For all functions \(v_0 \in H_0^1(\Omega_m)\) and \(v \in [H_0^1(\Omega_m)]^3\), the Clément interpolation operator \(I_{\mathrm{Cl}}^0\) satisfies the following estimates:
		\begin{align}
			\sum_{T\in \mathcal{T}} h_{\min,T}^{-2} \| v_0 - I_{\mathrm{Cl}}^0 v_0\|_{T}^2 &\lesssim m_1^2(v_0,\mathcal{T}_{ph}) \| \nabla v_0\|^2; \quad \label{29}\\
			\sum_{F\in \mathcal{F}} h_{\min,F}^2 \| v_0 - I_{\mathrm{Cl}}^0 v_0\|_{F}^2 &\lesssim m_1^2(v_0,\mathcal{F}) \| \nabla v_0\|^2; \quad \label{30}\\
			\sum_{F\in \mathcal{F}_{ph}} h_{\min,F}^2 \| v - I_{\mathrm{Cl}} v\|_{F}^2 &\lesssim m_1^2(v,\mathcal{F}_{ph}) \| v\|_{H_0^1(\Omega_m)}^2. \quad \label{31}
		\end{align}
	\end{theorem}
	\begin{proof}
		See \cite{kunert2000posteriori}.
	\end{proof}
	
	\subsection{Interpolation in \(W_{ph}\)}
	
	Since we work on a cartesian product spaces, it is important to define an operator to project the functions from \(W\) into \(W_{ph}\). Hence, we define the following interpolation operator.
	
	\begin{definition}
		For any \(\nu \in W\), its interpolant \(I_{ph}^0 \nu \in W_{ph}\) is defined by:
		\[
		I_{ph}^0 \nu = \left(I_{\mathrm{Cl}}^0 \nu^m, I_L \nu^c\right).
		\]
	\end{definition}
	
	\section{A Priori Error Analysis}\label{sec5}
	
	We define the total error at time \(t_p\) by \(u(t_p) - u_{ph}\) which can be decomposed as
	\[
	u(t_p) - u_{ph} = u(t_p) - u_p + u_p - u_{ph}
	\]
	where \(\theta_p := u(t_p) - u_p\) is the time discretization error, and \(\epsilon_p := u_p - u_{ph}\) is the space discretization error.
	
	\subsection{A-priori analysis of the time discretization}
	
	Let's introduce the linear functional \(F:[0,T]\times W \longrightarrow W^{*}\) such that the problem (\ref{model3}) is equivalent to:
	\begin{equation}
		\left\{
		\begin{array}{ll}
			S(\partial_t u^m(t,\cdot),w) = (F(t,u(t,\cdot)),w), & \forall w^m \in H_0^1(\Omega_m), \forall t\in [0,T]\\[5pt]
			a^c(u(t,\cdot),w^c) = (R^c(t,\cdot),w^c), & \forall w^c \in H_0^1(\Omega_c)
		\end{array}
		\right. \quad \label{32}
	\end{equation}
	where \(W^{*}\) is the dual space of \(W\).
	
	With the above notation, \(F\) is defined as:
	\[
	(F(t,\nu),w^m) = -a^m(\nu,w^m) + (R^m(t,\cdot),w^m), \quad \forall (\nu,w)\in W^2.
	\]
	
	The semi-discrete problem (\ref{19}) is equivalent to the following problem (\ref{33}) consisting in finding a sequence \((u_p)_{0\leq p\leq N}\) such that:
	\begin{equation}
		\left\{
		\begin{array}{ll}
			S\left(\dfrac{u_p^m - u_{p-1}^m}{\tau_p}, w^m\right) = (F(t_p,u_p),w^m), & \forall w^m \in H_0^1(\Omega_m)\\[10pt]
			a^c(u_p,w^c) = (R^c(t_p,\cdot),w^c), & \forall w^c \in H_0^1(\Omega_c)
		\end{array}
		\right. \quad \label{33}
	\end{equation}
	
	\begin{theorem}\label{thm2}
		Assume that the exact solution \(u = (u^m,u^c)\) satisfies:
		\[
		u^m \in L^2(0,T;H^2(\Omega_m)) \cap H^2(0,T;L^2(\Omega_m)),\quad u^c \in L^2(0,T;H^2(\Omega_c)).
		\]
		For any \(n\in [1,N]\), we have:
		\[
		S\| \theta_n^m\|^2 + \sum_{p=1}^n \tau_p \| \theta_p^m\|_{H_0^1(\Omega_m)}^2 \lesssim \tau^2 \| \partial_t^2 u\|_{H^2(0,t_n;L^2(\Omega_m))}^2,
		\]
		where \(\tau = \max_{1\leq p\leq N} \tau_p\).
	\end{theorem}
	
	\begin{proof}
		Subtracting (\ref{32}) from (\ref{33}), we obtain:
		\[
		S\left(\partial_t u^m(t_p,\cdot) - \frac{u_p^m - u_{p-1}^m}{\tau_p}, w^m\right) = \left(F(t_p,u(t_p,\cdot)) - F(t_p,u_p), w\right). \quad \label{34}
		\]
		
		Using Taylor expansion around \(t_{p-1}\), we have:
		\[
		u^m(t_p,\cdot) = u^m(t_{p-1},\cdot) + \tau_p \partial_t u^m(t_{p-1},\cdot) + \frac{\tau_p^2}{2} \partial_t^2 u^m(\xi_p,\cdot),
		\]
		for some \(\xi_p \in [t_{p-1},t_p]\).
		
		Thus, we have
		\[
		\partial_t u^m(t_{p-1}) = \frac{u^m(t_p) - u^m(t_{p-1})}{\tau_p} - \frac{\tau_p}{2} \partial_t^2 u^m(\xi_p,\cdot),
		\]
		for some \(\xi_p \in [t_{p-1},t_p]\).
		
		\[
		S\left(\frac{\theta_p^m - \theta_{p-1}^m}{\tau_p} - \frac{\tau_p}{2} \partial_t^2 u^m(\xi_p,\cdot), w^m\right) = -a^m\left(u(t_p,\cdot) - u_p, w^m\right).
		\]
		
		Setting \(w^m = \theta_p^m\) gives the following:
		\[
		\frac{S\tau_p}{2}\left(\|\theta_p^m\|^2 - \|\theta_{p-1}^m\|^2 + \|\theta_p^m - \theta_{p-1}^m\|^2\right) + a^m(\theta_p,\theta_p^m) = \frac{S\tau_p}{2}\|\partial_t^2 u(\xi_p,\cdot)\| \|\theta_p^m\|.
		\]
		
		From the coercivity of \(a^m\), we have: \(\|\theta_p^m\|_{H_0^1(\Omega)}^2 \lesssim a^m(\theta_p,\theta_p^m)\).
		
		Applying Young's inequality, we have:
		\[
		\|\partial_t^2 u(\xi_p,\cdot)\| \|\theta_p^m\| \leq \frac{S\tau_p}{4} \|\partial_t^2 u(\xi_p,\cdot)\|^2 + \frac{1}{S\tau_p} \|\theta_p^m\|^2.
		\]
		
		Thanks to Poincaré inequality, the relation yields:
		\begin{eqnarray}
		\frac{S\tau_p}{2}\left(\|\theta_p^m\|^2 - \|\theta_{p-1}^m\|^2\right) + \frac{1}{2}\|\theta_p^m\|_{H_0^1(\Omega_m)}^2 \lesssim \frac{S^2\tau_p^2}{8} \|\partial_t^2 u(\xi_p,\cdot)\|^2. \quad \label{36}
		\end{eqnarray}
		
		Multiplying (\ref{36}) by \(2\tau_p\) and summing over \(p\in [1,n]\), we obtain:
		\[
		S\|\theta_p^m\|^2 + \sum_{p=1}^n \tau_p \|\theta_p^m\|_{H_0^1(\Omega_m)}^2 \lesssim \sum_{p=1}^n \tau_p^3 \|\partial_t^2 u(\xi_p,\cdot)\|^2. \quad \square
		\]
		
	\end{proof}
	
	\subsection{A-priori analysis of the space discretization}
	
	The finite element methods presented above are conforming for approximating the hydraulic head in the conduit \(\Omega_c\) and non-conforming for approximating the hydraulic head in the matrix \(\Omega_m\). Thus, we will perform the a-priori error analysis separately and then combine them. As \(a_p\) is a bilinear form, the problem (\ref{19}) is equivalent to the following problem:
	\begin{equation}
		\left\{
		\begin{array}{ll}
			A_p(u_p^m,w) = S(u_{p-1}^m,w)_{\Omega_m} + \tau_p(R_p^m,w)_{\Omega_m} + \tau_p \alpha_{ex}(u_p^c,\Pi w)_{\Omega_c} & \forall w \in W_{ph}^m\\[5pt]
			B(u_p^c,w) = (R_p,w) + \alpha_{ex}(\Pi u_p^m,w) & \forall w \in W_{ph}^c
		\end{array}
		\right., \quad \label{37}
	\end{equation}
	where \(A_p\) and \(B\) are two symmetric bilinear forms respectively defined on \(H_0^1(\Omega_m)\) and \(H_0^1(\Omega_c)\) by:
	\begin{equation}
		\begin{array}{l}
			A_p(v,w) = S\displaystyle\int_{\Omega_m} v w + \tau_p K\int_{\Omega_m} \nabla v \cdot \nabla w + \tau_p \alpha_{ex}\int_a^b \Pi(v w)dx,\\[10pt]
			B(v,w) = D\displaystyle\int_a^b \partial_x v \cdot \partial_x w + \alpha_{ex}\int_a^b v w.
		\end{array} \quad \label{38}
	\end{equation}
	
	Similarly, the problem (\ref{20'}) is equivalent to the following problem:
	\begin{equation}
		\left\{
		\begin{array}{ll}
			A_{ph}(u_{ph}^m,w_{ph}^m) = S(u_{p-1}^m,w_{ph}^m) + \tau_p(R_{ph}^m,w_{ph}^m) + \tau_p \alpha_{ex}(u_{ph}^c,\Pi w_{ph}^m)_{\Omega_c} & \forall w^m \in W_{ph}^m\\[5pt]
			B_h(u_{ph}^c,w_{ph}^c) = (R_p^c,w_{ph}^c) + \alpha_{ex}(\Pi u_{ph}^m,w_{ph}^c) & \forall w_{ph}^c \in W_{ph}^c,
		\end{array}
		\right. \quad 
	\end{equation}
	where \(A_{ph}\) and \(B_h\) are two symmetric bilinear forms respectively defined on \(W_{ph}^m\) and \(W_{ph}^c\) by:
	\[
	A_{ph}(v,w) = \sum_{T\in \mathcal{T}_{ph}}\int_T(S v w + \tau_p K\nabla v \cdot \nabla w) + \tau_p \alpha_{ex}\sum_{E\in \mathcal{T}_{ph}}\int_E \Pi(v w)dx,
	\]
	\[
	B_h(v,w) = \sum_{E\in \mathcal{T}_{ph}}\int_E(D\partial_x v \cdot \partial_x w + \alpha_{ex} v w).
	\]
	
	\begin{proposition}
		Let \(C_{\Omega_m}\) and \(C_{ab}\) be the Poincaré constants associated respectively to the domain \(\Omega_m\) and \((a,b)\). Let \(p\in [1,N]\).
		\begin{enumerate}
	\item 	For any \(v^m \in H_0^1(\Omega_m)\) and \(w^m \in H_0^1(\Omega_m)\), we have:
		\begin{equation}
			|A_{ph}(v^m,w^m)| \leq (SC_{\Omega_m}^2 + \tau_p K + \tau_p \alpha_{ex} C_{\Pi}^2) \| v^m\|_{H_0^1(\Omega_m)} \| w^m\|_{H_0^1(\Omega_m)} \quad \label{39}
		\end{equation}
		
		\item If \(\tau_p K - SC_{\Omega_m}^2 - \tau_p \alpha_{ex} C_{\Pi}^2 > 0\), then the bilinear form \(A_p\) is coercive, i.e for any \(v^m \in H_0^1(\Omega_m)\), we have:
		\begin{equation}
			A_p(v^m,v^m) \geq \left(\tau_p K - SC_{\Omega_m}^2 - \tau_p \alpha_{ex} C_{\Pi}^2\right) \| v^m\|_{H_0^1(\Omega_m)}^2. \label{40}
		\end{equation}
		
		\item For any \(v^c \in H_0^1(a,b)\) and \(w^c \in H_0^1(a,b)\), we have:
		\begin{equation}
			|B(v^c,w^c)| \leq (D + \alpha_{ex} C_{ab}^2) \| v^c\|_{H_0^1(a,b)} \| w^c\|_{H_0^1(a,b)} \quad \label{41}
		\end{equation}
		
		\item If \(D - \alpha_{ex} C_{ab}^2 > 0\), then the bilinear form \(B\) is coercive, i.e for any \(v^c \in H_0^1(a,b)\), we have:
		\begin{equation}
			B(v^c,v^c) \geq \left(D - \alpha_{ex} C_{ab}^2\right) \| v^c\|_{H_0^1(a,b)}. \quad \label{42}
		\end{equation}
	\item 	The same results hold for the bilinear forms \(A_{ph}\) and \(B_h\) and involve the same constants.
	\end{enumerate}
	\end{proposition}
	\begin{remark}
	For any \(p\in [1,N]\), let's define the following constants:
	\[
	\alpha_p = \tau_p K - SC_{\Omega_m}^2 - \tau_p \alpha_{ex} C_{\Pi}^2,\quad c_p = SC_{\Omega_m}^2 + \tau_p K + \tau_p \alpha_{ex} C_{\Pi}^2,
	\]
	\[
	\beta_p = D - \alpha_{ex} C_{ab}^2,\quad d_p = D + \alpha_{ex} C_{ab}^2.
	\]
	
	For the remainder of this analysis, for any \(p\in [1,N]\), we suppose that:
	\begin{equation}
		\alpha_p > 0 \quad \text{and} \quad \beta_p > 0. \quad \label{43}
	\end{equation}
	\end{remark}
	\begin{theorem}\mbox{ }\label{th}
		\begin{enumerate}
			\item For any \(n\in [0,N]\), we have:
			\begin{equation}
				\| e_p^c\|_{H_0^1(\Omega_m)} \lesssim h_c \| u_p^c\|_{H_0^1(\Omega_c)}. \quad \label{44}
			\end{equation}
			\item For any \(n\in [0,N]\), we have:
			\begin{equation}
				\| e_n^m\|_{H_0^1(\Omega_m)} \lesssim \left(\prod_{j=1}^n \rho_j\right) \| e_0^m\|_{H_0^1(\Omega_m)} + \sum_{k=1}^n \left(\prod_{j=k+1}^n \rho_j\right) \Psi_k, \quad \label{45}
			\end{equation}
		\end{enumerate}
		where we used the following notation:
		\[
		\rho_j = \frac{SC_{\Omega_m}^2}{\alpha_j},\quad \Psi_k = \frac{\tau_k}{\alpha_k}\left(\alpha_{ex} C_{ab} C_{\Pi} h_c \| u_k^c\|_{H_0^1(a,b)} + K h_m m_1(u_k^m,\mathcal{T}_{kh}) \| u_k^m\|_{H_0^1(\Omega_m)}\right).
		\]
	\end{theorem}
	
	\begin{proof}\mbox{ }
		\begin{enumerate}
			\item Recall that \(a^c\) is continuous and coercive. From Céa's lemma \cite{ern2005aide}, we have:
			\[
			\| u_p^c - u_{ph}^c\|_{H_0^1(\Omega_c)} \lesssim \inf_{v_h \in W_{ph}^c} \| u_p^c - v_h\|_{H_0^1(\Omega_c)} \leq \| u_p^c - I_L u_p^c\|_{H_0^1(\Omega_c)}.
			\]
			
			By applying the interpolation estimate (\ref{23}), we have (\ref{44}).
			
			\item Now, let's consider \(v_{ph}^m \in W_{ph}^m\) and \(w_{ph}^m \in W_{ph}^m\). By linearity, we have:
			\[
			A_{ph}(u_{ph}^m - v_{ph}^m,w_{ph}^m) = A_{ph}(u_{ph}^m - u_{ph}^m,w_{ph}^m) + A_{ph}(u_{ph}^m - v_{ph}^m,w_{ph}^m)
			\]
			
			As \(A_{ph}\) is continuous, we have:
			\begin{eqnarray*}
			A_{ph}(u_{ph}^m - v_{ph}^m,w_{ph}^m) &\leq& \sup_{w_{ph}^m \in W_{ph}^m}\left\{ \frac{|A_{ph}(u_{ph}^m,w_{ph}^m) - A_p(u_{ph}^m,w_{ph}^m)|}{\|w_{ph}^m\|_{W_{ph}^m}}\right\} \| w_{ph}^m\|_{W_{ph}^m} \\
			&+& c_p \| w_{ph}^m - v_{ph}^m\|_{H_0^1(\Omega_m)} \| w_{ph}^m\|_{W_{ph}^m}.
			\end{eqnarray*}
			
			Using the coercivity of \(A_{ph}\), we have:
			\begin{eqnarray*}
			\alpha_p \| u_{ph}^m - v_{ph}^m\|_{H_0^1(\Omega_m)} &\leq& \sup_{w_{ph}^m \in W_{ph}^m} \frac{|A_{ph}(u_{ph}^m,w_{ph}^m) - A_p(u_{ph}^m,w_{ph}^m)|}{\|w_{ph}^m\|_{W_{ph}^m}}\\
			& &+ c_p \inf_{v_{ph}^m \in W_{ph}^m} \| u_{ph}^m - v_{ph}^m\|_{H_0^1(\Omega_m)}.
			\end{eqnarray*}
			
			Using the triangular inequality, we have:
			\begin{eqnarray*}
			\alpha_p \| u_{ph}^m - u_{ph}^m\|_{H_0^1(\Omega_m)} &\leq& \sup_{w_{ph}^m \in W_{ph}^m} \frac{|A_{ph}(u_{ph}^m,w_{ph}^m) - A_p(u_{ph}^m,w_{ph}^m)|}{\|w_{ph}^m\|_{W_{ph}^m}} \\
			&+ &(\alpha_p + c_p) \inf_{v_{ph}^m \in W_{ph}^m} \| u_{ph}^m - v_{ph}^m\|_{H_0^1(\Omega_m)}.
			\end{eqnarray*}
			
			\[
			\inf_{v_{ph} \in W_{ph}^m} \| u_{ph}^m - v_{ph}\|_{H_0^1(\Omega_m)} \lesssim h_m m_1(u_{ph}^m,\mathcal{T}_{ph}) \| u_{ph}^m\|_{H_0^1(\Omega_m)}.
			\]
			
			Subtracting the first equations of the problems (\ref{37}) and (\ref{38}), we obtain:
			\[
			A_{ph}(u_{ph}^m,w_{ph}) - A_p(u_p,w_{ph}) = S(u_{p-1,h}^m - u_{p-1}^m,w_{ph}^m) + \alpha_{ex}(u_{ph}^c - u_p^c,\Pi w_{ph}^m).
			\]
			
			Subsequently, applying Cauchy-Schwarz and Poincaré inequalities yield to the following estimations:
			\[
			|(u_p^c - u_{ph}^c,\Pi w_{ph}^m)| \leq C_{ab} C_{\Pi} \| u_p^c - u_{ph}^c\|_{H_0^1(a,b)} \times \| w_{ph}^m\|_{W_{ph}^m}
			\]
			\[
			|(u_{p-1}^m - u_{p-1,h}^m,w_{ph}^m)| \leq C_{\Omega_m}^2 \| u_p^m - u_{p-1,h}^m\|_{H_0^1(\Omega_m)} \| w_{ph}^m\|_{W_{ph}^m}
			\]
			
			Thus the inequality is equivalent to:
			\begin{eqnarray}\nonumber
				\alpha_p \| e_p^m\|_{H_0^1(\Omega_m)} &\lesssim& S C_{\Omega_m}^2 \| e_{p-1}^m\|_{H_0^1(\Omega_m)} + \tau_p \alpha_{ex} C_{ab} C_{\Pi} \| e_p^c\|_{H_0^1(a,b)}\\ &+& \tau_p K h_m m_1(u_p^m,\mathcal{T}_{ph}) \| u_p^m\|_{H_0^1(\Omega_m)} \quad \label{47}
			\end{eqnarray}
			
			With the estimation (\ref{44}), we have:
			\begin{eqnarray}\nonumber
				\| e_p^m\|_{H_0^1(\Omega_m)} &\lesssim& \frac{S C_{\Omega_m}^2}{\alpha_p} \| e_{p-1}^m\|_{H_0^1(\Omega_m)}\\
				& +& \frac{\tau_p}{\alpha_p} \left( \alpha_{ex} C_{ab} C_{\Pi} h_c \| u_p^c\|_{H_0^1(a,b)} + K h_m m_1(u_p^m,\mathcal{T}_{ph}) \| u_p^m\|_{H_0^1(\Omega_m)} \right) \quad \label{48}
			\end{eqnarray}
			
			Applying recursively this estimate, for any \(n\in [1,N]\), we obtain (\ref{45}).
		\end{enumerate}
	\end{proof}
	
	\begin{remark}[Stability Analysis and Physical Interpretation]
		The error estimate in Theorem \ref{th} reveals several important aspects of the numerical method's behavior:
		
		\begin{itemize}
			\item [i)]\textbf{Stability Conditions:} The decay or boundedness of the error over time depends critically on the stability parameter \(\rho_j = \dfrac{SC_{\Omega_m}^2}{\alpha_j}\). For the method to be stable, we need \(\rho_j \leq 1\) for all \(j\), which translates to:
			\[
			\alpha_j = \tau_j K - SC_{\Omega_m}^2 - \tau_j \alpha_{ex} C_{\Pi}^2 \geq SC_{\Omega_m}^2
			\]
			This gives the stability condition:
			\begin{equation}
				\tau_j \geq \frac{2SC_{\Omega_m}^2}{K - \alpha_{ex} C_{\Pi}^2}. \quad \label{49}
			\end{equation}
			
			\item [ii)]\textbf{Physical Parameter Influence:} The stability depends on the interplay between:
			\begin{itemize}
				\item Storage coefficient \(S\) (larger \(S\) requires larger time steps)
				\item Hydraulic conductivity \(K\) (larger \(K\) improves stability)
				\item Exchange coefficient \(\alpha_{ex}\) (larger exchange requires more careful time stepping)
				\item Geometric properties through Poincaré constants
			\end{itemize}
			
			\item[iii)] \textbf{Asymmetric Behavior:} The error behavior differs significantly between the two domains:
			\begin{itemize}
				\item Conduit error \((\Omega_c)\): Depends only on mesh size \(h_c\) and exhibits standard finite element convergence.
				\item Matrix error \((\Omega_m)\): Involves complex temporal propagation through the product \(\prod_j \rho_j\).
			\end{itemize}
			
			\item[iv)] \textbf{Temporal Error Propagation:} When \(\rho_j > 1\), the error can grow exponentially as \(\prod_{j=1}^n \rho_j \sim \rho^n\) where \(\rho > 1\). This exponential growth can render the numerical method impractical for long-time simulations.
			
			\item[v)] \textbf{Coupling Effects:} The coupling between matrix and conduit through the exchange term \(\alpha_{ex}\) affects both the stability condition and the error propagation, making the analysis more complex than for decoupled problems.
			
			\item [vi)]\textbf{Adaptive Time Stepping:} The stability condition suggests that adaptive time stepping strategies should monitor the ratio \(\dfrac{SC_{\Omega_m}^2}{\alpha_j}\) and adjust \(\tau_j\) accordingly to maintain stability.
		\end{itemize}
	\end{remark}
	
	\begin{corollary}\label{co1}
		For any \(n\in [1,N]\), under the stability condition (\ref{49}), the space error satisfies:
		\begin{equation}
			\| e_n\|_{W} \lesssim \| e_0^m\|_{H_0^1(\Omega_m)} + \sum_{k=1}^n \tau_k \left( h_c + h_m m_1(u_k^m,\mathcal{T}_{kh}) \right) \| u_k\|_{W}, \quad \label{50}
		\end{equation}
		where the underlying constant depends on the physical parameters.
		
		Moreover, if the meshes \(\mathcal{T}_{kh}\) are uniform, we have:
		\begin{equation}
			\| e_n\|_{W} \lesssim \| e_0^m\|_{H_0^1(\Omega_m)} + (h_c + h_m) \sum_{k=1}^n \tau_k \| u_k\|_{W}. \quad \label{51}
		\end{equation}
	\end{corollary}
		\section*{Corollary (Full Discretization A-priori Estimation)}
	Under the assumptions of Theorem \ref{thm2} and Corollary \ref{co1}, the coercivity assumptions (\ref{43}), for any \(n\in [1,N]\), the error satisfies:
	\begin{equation*}
	\| u(t_n) - u_{nh}\|_{W} \lesssim \tau \| \partial_t^2 u\|_{H^2(0,t_n;L^2(\Omega_m))} + \sum_{k=1}^n \tau_k \left[ h_c + h_m m_1(u_k^m,\mathcal{T}_{kh}) \right] \| u_k\|_W + \| e_0^m\|_{H_0^1(\Omega_m)}. \quad \label{52}
	\end{equation*}
	
	Moreover, if the meshes \(\mathcal{T}_{kh}\) are uniform, we have:
	\begin{equation*}
	\| u(t_n) - u_{nh}\|_{W} \lesssim \tau \| \partial_t^2 u\|_{H^2(0,t_n;L^2(\Omega_m))} + (h_c + h_m) \sum_{k=1}^n \tau_k \| u_k\|_W + \| e_0^m\|_{H_0^1(\Omega_m)}. \quad \label{53}
	\end{equation*}
	\section{A-posteriori Error Analysis}\label{sec6}
	
	For a given sequence \((v_p^m)_{0\leq p\leq N}\) in \(H_0^1(\Omega_m) \oplus W_h^m\), denote by \(v_t^m\) its Lagrange interpolant which is affine on each interval \([t_{p-1},t_p]\), \(1\leq p\leq N\), and equal to \(v_p^m\) at \(t_p\), i.e defined by:
	\[
	\forall t\in [t_{p-1},t_p], \quad v_t^m(t,\cdot) = \frac{t_p - t}{\tau_p} v_{p-1}^m + \frac{t - t_{p-1}}{\tau_p} v_p^m.
	\]
	
	For \(s\in \{m,c\}\) and a given sequence \((v_p^s)_{1\leq p\leq N}\) in \(H_0^1(\Omega_s)\), denote by \(\pi_t u^s\) the step function which is constant and equal to \(v^s(t_p)\) on each interval \((t_{p-1},t_p)\).
	
	\subsection{A-posteriori Error Analysis in Time}
	
	The time discretization error is denoted by: \(e_t = (u^m,u^c) - (u_t^m,\pi_t u^c)\) and satisfies the residual equation. Subtracting the equation of Theorem \ref{thm1} from (\ref{19}), we obtain the following result.
	
	\begin{proposition}
		For any \(t\in (t_{p-1},t_p)\) and \(w\in W\), we have:
		\begin{equation}
			\begin{array}{rl}
				& S(\partial_t(e_t^m(t,\cdot),w^m) + a^m(e_t,w^m) + a^c(e_t,w^c) = (R(t,\cdot) - \pi_t R(t,\cdot),w) \\[5pt]
				& + K\displaystyle\int_{\Omega_m} \nabla \left(u_t^m(x) - u_t^m(t,x)\right) \cdot \nabla w^m(x) \\[5pt]
				& +\dfrac{\alpha_{ex}}{2\pi}\displaystyle\int_a^b\int_0^{2\pi} \left(u_t^m(s_0(x,\theta)) - u_t^m(t,s_0(x,\theta))\right) \cdot \left(w^m(s_0(x,\theta)) - w^c(x)\right) d\theta dx.
			\end{array} \quad \label{54}
		\end{equation}
	\end{proposition}

	Note that: $$u_p^m - u_t^m(t,\cdot) = \dfrac{t_p - t}{\tau_p} (u_p^m - u_{p-1}^m).$$ In analogy to \cite{bernardi2004posteriori, johnson1990posteriori, nicaise2005posteriori, scott1990finite}, let's define time local indicator by:
	\[
	\eta_p = \tau_p^{1/2} \| \nabla (u_{ph}^m - u_{p-1,h}^m)\|_{L^2(\Omega_m)}.
	\]
	In contrast to the approaches in \cite{bernardi2004posteriori, johnson1990posteriori, scott1990finite}, the difference \(u_{ph}^m - u_{p-1,h}^m\) does not belong to \(H^1(\Omega_m)\). Since the discrete approximation spaces \(W_{ph}^m\) are not used in the formulation of the continuous problems (\ref{model3}) and (\ref{19}), we follow the configuration proposed in \cite{nicaise2005posteriori} for that.
	Moreover, the difference \(u_{ph}^m - u_{p-1,h}^m\) involves two distinct meshes \(\mathcal{T}_{ph}\) and \(\mathcal{T}_{p-1,h}\). Thus we shall see that difference as a piecewise \(P_k\) function on the mesh \(\mathcal{T}_{ph} \cap \mathcal{T}_{p-1,h}\) which simply is made of the intersections of elements of \(\mathcal{T}_{ph}\) with ones of \(\mathcal{T}_{p-1,h}\). The broken gradient \(\nabla (u_{ph}^m - u_{p-1,h}^m)\) is then evaluated on this latter mesh.
	
	These local indicators satisfy the following residual-based a posteriori estimate.
	
	\begin{theorem}[Upper Bound on Time Error]\label{thm3}
		For any \(n\in \{1,2,\ldots ,N\}\), one has:
		\begin{equation}
			\begin{array}{rl}
				& S\| e_t^m(t_n,\cdot)\|_{L^2(\Omega_m)}^2 + K\| e_t^m\|_{L^2(0,t_n;H_0^1(\Omega_m))}^2 + D\| e_t^c\|_{L^2(0,t_n;H_0^1(a,b))}^2 \leq \left(\displaystyle\sum_{p=1}^n \eta_p^2\right) \\[8pt]
				& \quad + \| R^m - \pi_t R^m\|_{L^2(0,t_n;H^{-1}(\Omega_m))}^2 + \| R^c - \pi_t R^c\|_{L^2(0,t_n;H^{-1}(a,b))}^2 \\[8pt]
				&\quad + \| u_t^m - u_{ht}^m\|_{L^2(0,t_n;H_0^1(\Omega_m))}^2,
			\end{array} \quad \label{55}
		\end{equation}
	\end{theorem}
	
	\begin{proof}
		Taking \(w = e_t(t,\cdot)\) in (\ref{54}), thanks to Cauchy-Schwarz, Young and Poincaré inequalities applied subsequently, we obtain:
		\[
		\begin{array}{rl}
			& \dfrac{S}{2}\dfrac{d}{dt}\| e_t(t,\cdot)\|_{L^2(\Omega_m)}^2 + \dfrac{K}{2}\| e_t^m(t,\cdot)\|_{H_0^1(\Omega_m)}^2 + \dfrac{D}{2}\| e_t^c(t,\cdot)\|_{H_0^1(a,b)}^2 \\[10pt]
			& \leq \dfrac{4}{K}\| R^m(t,\cdot) - R_p^m\|_{H^{-1}(\Omega_m)}^2 + \dfrac{2}{D}\| R^c(t,\cdot) - R_p^c\|_{H^{-1}(a,b)}^2 \\[10pt]
			& \quad + \left(4K + \dfrac{\alpha_{ex}}{2}\right) \| u_p^m - u_t^m(t,\cdot)\|_{H_0^1(\Omega_m)}^2
		\end{array} \quad \label{56}
		\]
		
		Integrating over \((t_{p-1},t_p)\) and summing over \(p\in [1,n]\) yields:
		\[
		\begin{array}{rl}
			& S\| e_t^m(t_n,\cdot)\|_{L^2(\Omega_m)}^2 + K\| e_t^m(t,\cdot)\|_{L^2(0,t_n;H_0^1(\Omega_m))}^2 + D\| e_t^c(t,\cdot)\|_{L^2(0,t_n;H_0^1(a,b))}^2 \\[5pt]
			& \lesssim \dfrac{1}{K}\| R^m(t,\cdot) - R_p^m\|_{L^2(0,t_n;H^{-1}(\Omega_m))}^2 + \dfrac{1}{D}\| R^c(t,\cdot) - R_p^c\|_{L^2(0,t_n;H^{-1}(a,b))}^2 \\[5pt]
			& \quad + \left(4K + \dfrac{\alpha_{ex}}{2}\right) \sum_{p=1}^n \tau_p \| u_p^m - u_{p-1}^m\|_{H_0^1(\Omega_m)}^2.
		\end{array} \quad 
		\]
		
		With triangular inequality and Young's inequality, we have:
		\[
		\tau_p \| u_p^m - u_{p-1}^m\|_{H_0^1(\Omega_m)}^2 \lesssim \eta_p^2 + \tau_p \| u_p^m - u_{p,h}^m\|_{H_0^1(\Omega_m)}^2 + \tau_p \| u_{p-1}^m - u_{p-1,h}^m\|_{H_0^1(\Omega_m)}^2.
		\]
		
		As we have \((u_t^m - u_h^m)(t,\cdot) = \frac{t - t_{p-1}}{\tau_p}(u_p^m - u_{p,h}^m) + \frac{t_p - t}{\tau_p}(u_{p-1}^m - u_{p-1,h}^m)\), integrating the square of its norm and using Cauchy-Schwarz inequality and Young's inequality lead to:
		\begin{eqnarray}
		\tau_p \| u_p^m - u_{p,h}^m\|_{H_0^1(\Omega_m)}^2 + \tau_p \| u_{p-1}^m - u_{p-1,h}^m\|_{H_0^1(\Omega_m)}^2 \lesssim \int_{t_{p-1}}^{t_p} \| (u_t^m - u_h^m)(t,\cdot)\|_{H_0^1(\Omega_m)}^2 dt. \quad \label{56'}
		\end{eqnarray}
		This completes the proof.
			\end{proof}
		
		As a consequence of the upper bound in Theorem \ref{thm3}, we obtain a bound on the time derivative of the error in the dual norm.

	\begin{corollary}
		For any \(n \in [1, N]\), we have:
		\begin{equation}
			\begin{array}{rl}
				& \| \partial e_t^m\|_{L^2(0,t_n;H^{-1}(\Omega_m))} \lesssim \left( \displaystyle\sum_{p=1}^n \eta_p^2\right) + \| R^m - \pi_t R^m\|_{L^2(0,t_n;H^{-1}(\Omega_m))}^2 + \| R^c - \pi_t R^c\|_{L^2(0,t_n;H^{-1}(a,b))}^2  \\[5pt]
				& \qquad  + \| u_t^m - u_h^m\|_{L^2(0,t_n;H_0^1(\Omega_m))}^2.
			\end{array} \quad \label{57}
		\end{equation}
	\end{corollary}
	
	\begin{proof}
		For \(t \in (t_{p-1}, t_p)\), we have:
		\[
		\| \partial_t e_t^m(t,\cdot)\|_{H^{-1}(\Omega_m)} = \sup_{\nu \in H_0^1(\Omega_m)} \frac{(\partial_t e_t^m(t,\cdot),\nu^m)}{\| \nu^m\|_{H_0^1(\Omega_m)}}.
		\]
		
		Taking \(e_t\) as \((e_t^m, 0)\) in the time residual equation (\ref{54}), for any \(\nu \in H_0^1(\Omega_m)\), we obtain:
		\begin{eqnarray}
		\begin{array}{rl}
			& S(\partial_t e_t^m(t,\cdot),\nu^m) = \left(R^m(t,\cdot) - \pi_t R^m,\nu^m\right) + K\displaystyle\int_{\Omega_m} \nabla (u_p^m(x) - u_t^m(t,\cdot)) \cdot \nabla \nu^m(x) dx \\[5pt]
			& + \alpha_{ex}\displaystyle\int_a^b \Pi \left((u_p^m - u_t^m(t,\cdot))\nu^m\right)(x) dx - K(\nabla e_t^m(t,\cdot),\nabla \nu^m) - \alpha_{ex}\int_a^b \Pi \left(e_t^m(t,\cdot)\nu^m\right)(x) dx
		\end{array} \quad \label{54'}
		\end{eqnarray}
		
		Applying Cauchy-Schwarz and Poincaré inequalities and the estimate (A.2), we have:
		\[
		S(\partial_t e_t^m(t,\cdot),\nu^m) \lesssim \| R^m(t,\cdot) - R_p^m\|_{H^{-1}(\Omega_m)} \| \nu^m\|_{H_0^1(\Omega_m)} + (K + \alpha_{ex}) \| e_t^m(t,\cdot)\|_{H_0^1(\Omega_m)} \| \nu^m\|_{H_0^1(\Omega_m)}
		\]
		\[
		\qquad + (K + \alpha_{ex}) \| u_p^m - u_t^m(t,\cdot)\|_{H_0^1(\Omega_m)} \| \nu^m\|_{H_0^1(\Omega_m)}
		\]
		
		The proof ends with Young's inequality and the estimate (\ref{55}).
	\end{proof}
	
	The following result gives a reliability estimate from below, ensuring that the error indicator \(\eta_p\) reflects the true error.
	
	\begin{theorem}[Lower Bound on Time Error]
		For any \(p \in [1, N]\), we have:
		\begin{equation}
			\begin{array}{rl}
				\eta_p^2 \lesssim & \| \partial_t e_t^m\|_{L^2(t_{p-1},t_p,H^{-1}(\Omega_m))}^2 + \| e_t^m\|_{L^2(t_{p-1},t_p,H_0^1(\Omega_m))}^2 + \| e_t^c\|_{L^2(t_{p-1},t_p,H_0^1(a,b))}^2 \\[5pt]
				& + \| R - \pi_t R\|_{L^2(t_{p-1},t_p,H^{-1}(\Omega_m))}^2 + \| u_t^m - u_{h\tau}^m\|_{L^2(t_{p-1},t_p,H_0^1(\Omega_m))}^2.
			\end{array} \quad \label{58}
		\end{equation}
	\end{theorem}
	
	\begin{proof}
		Using triangular inequality and Young's inequality, we have:
		\[
		\eta_p^2 \leq \tau_p \| \nabla (u_p^m - u_{p,h}^m)\|_{L^2(\Omega_m)}^2 + \tau_p \| \nabla (u_{ph}^m - u_p^m)\|_{L^2(\Omega_m)}^2 + \tau_p \| \nabla (u_{p-1}^m - u_{p-1,h}^m)\|_{L^2(\Omega_m)}^2.
		\]
		
		Taking \((u_p^m - u_t^m(t,\cdot),0)\) as \(w\) in the residual equation (\ref{54}) leads to:
		\[
		\begin{array}{rl}
			& K\| \nabla (u_p^m - u_t^m(t,\cdot))\|_{L^2(\Omega_m)}^2 + \alpha_{ex}\displaystyle\int_a^b \Pi (u_p^m - u_t^m(t,\cdot)^2)(x) dx \\[5pt]
			= & -S\Big(\partial_t e_t^m(t,\cdot),u_p^m - u_t^m(t,\cdot)\Big) - K\Big(\nabla e_t^m(t,\cdot),\nabla (u_p^m - u_t^m(t,\cdot))\Big) \\[5pt]
			& - \alpha_{ex}\displaystyle\int_a^b \Pi \Big(e_t^m(u_p^m - u_t^m(t,\cdot))\Big)(x) dx + \Big(R(t,\cdot) - \pi_t R, u_p^m - u_t^m(t,\cdot)\Big)
		\end{array} \quad \label{59}
		\]
		
		Applying the Cauchy-Schwarz and Poincaré inequalities with the estimate (\ref{a2}), we obtain:
		\begin{eqnarray}\nonumber
		\frac{K}{2}\| \nabla (u_p^m - u_t^m(t,\cdot))\|_{L^2(\Omega_m)}^2 &\lesssim& \frac{S}{8K}\| \partial_t e_t^m(t,\cdot)\|_{H^{-1}(\Omega_m)}^2 + \left(\frac{K}{8} + \frac{\alpha_{ex}}{8K}\right)\| e_t^m(t,\cdot)\|_{H_0^1(\Omega_m)}^2 \\
		&+& \frac{1}{K}\| R(t,\cdot) - \pi_t R(t,\cdot)\|_{H^{-1}(\Omega_m)}^2
		\end{eqnarray}
		
		Since \(u_p^m - u_t^m(t,\cdot) = \frac{t_p - t}{\tau_p}(u_p^m - u_{p-1}^m)\) and \(\int_{t_{p-1}}^{t_p} \left(\frac{t_p - t}{\tau_p}\right)^2 = \frac{\tau_p}{3}\), integrating this latter inequality, we obtain:
		\begin{eqnarray}\nonumber
		\tau_p \| \nabla (u_p^m - u_{p-1}^m)\|_{L^2(\Omega_m)}^2 &\lesssim& \| \partial_t e_t^m\|_{L^2(t_{p-1},t_p,H^{-1}(\Omega_m))}^2 + \| e_t^m\|_{L^2(t_{p-1},t_p,H_0^1(\Omega_m))}^2\\
		& +& \| R - \pi_t R\|_{L^2(t_{p-1},t_p,H^{-1}(\Omega_m))}^2
		\end{eqnarray}
		
		Applying the estimate (\ref{56'}) leads to (\ref{58}).
	\end{proof}
	
	\subsection{A-posteriori Error Analysis in Space}
	
	The space discretization error is defined by \(e_{ph} = u_p - u_{ph}\).
	
	\subsubsection{A-posteriori Error Analysis in Pipe Conduit}
	
	For any \(w \in H_0^1(a,b)\), we have:
	\begin{equation}
		B(e_p^c,w) = \sum_{E\in \mathcal{P}_{ph}} G_p^E \cdot w + \alpha_{ex}(\Pi e_p^m,w), \quad \label{59'}
	\end{equation}
	where \(G_p^E\) is the residual given by:
	\[
	G_p^E = R_p^c + D\partial_{xx}^2 u_{ph}^c + \alpha_{ex}\left(u_{ph}^c - \Pi u_{ph}^m\right).
	\]
	
	Moreover, for any \(w_{ph} \in W_{ph}^c\), we have:
	\begin{equation}
		B(e_p^c,w_{ph}) = \alpha_{ex}(\Pi e_p^m,w_{ph}). \quad \label{60}
	\end{equation}
	
	As usual \cite{ainsworth2000posteriori}, the exact residuals \(G_p^E\) are replaced by an approximate element residuals. We denote respectively by \(\overline{G_p^E}\) the projection in convenient finite dimensional spaces of \(G_p^E\). This consists in taking the approximation of \(R_{ph}^c\) as projection of \(R_p^c\) in finite dimensional spaces since \(R_p^c\) is the only one that belongs to an infinite dimensional space. Usually, \(R_{ph}^c\) is considered as the mean value of \(R_p^c\), i.e
	\[
	R_{ph}^c = \frac{1}{|\Omega_c|}\int_{\Omega_c} R_p^c(x) dx.
	\]
	
	\begin{definition}
		Let \(p \geq 1\). The local error estimate \(\eta_p^E\) is defined by:
		\[
		\eta_E^p = h_E \| \overline{G}_E^p\|_E.
		\]
		
		The global spatial error estimator \(\eta_p^c\) in the pipe conduit is given by:
		\[
		\eta_p^c = \left( \sum_{E\in \mathcal{P}_{ph}} (\eta_E^p)^2 \right)^{1/2}.
		\]
		
		The local and global approximation terms are defined by:
		\[
		\xi_E^p = h_E \| R_p^c - R_{ph}^c\|_{\omega_E} \quad \text{and} \quad \xi_p^c = \left( \sum_{E\in \mathcal{P}_{ph}} (\xi_p^E)^2 \right)^{1/2}.
		\]
	\end{definition}
	
	\begin{theorem}\label{G1}
		For any \(p \in [1, N]\), we have:
		\begin{equation}
			\| e_p^c\|_{H_0^1(a,b)} \lesssim \eta_p^c + \zeta_p^c + \alpha_{ex} \| e_p^m\|_{H_0^1(\Omega_m)}. \quad \label{61}
		\end{equation}
		
		For any \(p \in [1, N]\) and \(E \in \mathcal{P}_{ph}\), we have also:
		\begin{equation}
			\eta_E^p \lesssim \xi_E^p + h_E \left( \alpha_{ex} \| e_p^m\|_E + D\| \partial_x e_p^c\|_E + \alpha_{ex} \| e_p^c\|_E \right). \quad \label{62}
		\end{equation}
	\end{theorem}
	
	\begin{proof} Consider $v_p\in H^1(a,b)$. For any $w_{ph}^c\in W_{ph}^c$, using the equations (\ref{59'}) and (\ref{60}), we have:
		\[
		B(e_p^c,\nu_p^c) = \sum_{E\in \mathcal{P}_{ph}} G_E^p \cdot (\nu_p^c - w_{ph}^c) + \alpha_{ex}(\Pi e_p^m,w_{ph}^c).
		\]
		
		Setting \(w_{ph}^c = I_L \nu_p^c\), applying Cauchy-Schwarz and the estimate (\ref{a1}) leads to:
		\[
		B(e_p^c,\nu_p^c) \lesssim \sum_{E\in \mathcal{P}_{ph}} h_E \| G_E^p\|_E \| \nu_p\|_{H_0^1(a,b)} + \alpha_{ex} \| e_p^m\|_{H_0^1(\Omega_m)} \| \nu_p\|_{H_0^1(a,b)}.
		\]
		
		Setting \(\nu_p^c = e_p^c\) and using the coercivity of \(B\), we get:
		\[
		\| e_p^c\|_{H_0^1(a,b)} \lesssim \sum_{E\in \mathcal{P}_{ph}} h_E \| G_E^p\|_E + \alpha_{ex} \| e_p^m\|_{H_0^1(\Omega_m)}.
		\]
		
		Applying triangular inequality and Cauchy-Schwarz inequality yield to (\ref{61}).
		
		Consider \(E\in \mathcal{P}_{ph}\) and \(b_E\) a one dimensional bubble function associated with segment element \(E\). Taking \(\nu_p = b_E \overline{G_E^p}\) in the residual equation (\ref{59'}), we obtain:
		\[
		\int_E \overline{G_E^p}^2 b_E = \int_E (\overline{G_E^p} - G_E^p) \cdot (\overline{G_E^p} b_E) + D\int_E \Pi e_p^m \cdot (\overline{G_E^p} b_E) + \alpha_{ex} \int_E \partial_x e_p^c \cdot \partial_x (\overline{G_E^p} b_E) + \alpha_{ex} \int_E e_p^c \cdot (\overline{G_E^p} b_E)
		\]
		
		Applying Cauchy-Schwarz inequality and similar result to Theorem 2.2 of \cite{ainsworth2000posteriori}, we obtain:
		\[
		\| \overline{G_E^p}\|_E^2 \lesssim \left( \| \overline{G_E^p} - G_E^p\|_E + \alpha_{ex} \| e_p^m\|_E + D \| \partial_x e_p^c\|_E + \alpha_{ex} \| e_p^c\|_E \right) \times \| \overline{G_E^p}\|_E
		\]
		
		As the interior residual \(\overline{G_E^p}\) does vanish on \(E\), we have (\ref{62}).
	\end{proof}
	
	\subsubsection{A-posteriori Error Analysis in Porous Matrix}
	
	For any \(w^m \in H_0^1(\Omega_m)\), we have:
	\begin{eqnarray}\nonumber
		A_p(e_p^m,w^m) &=& \tau_p \sum_{T\in \mathcal{T}_{ph}} \left[ \int_T G_T^p \cdot w^m + \int_{\partial T} \left(-K\frac{\partial w_{ph}^m}{\partial n_T}\right) \cdot w^m \right] + S(e_{p-1,h}^m,w^m) \\&+& \tau_p \alpha_{ex} \sum_{E\in \mathcal{P}_{ph}} \int_E e_p^c \Pi w^m \quad \label{63}
	\end{eqnarray}
	where \(n_T\) is unit outer normal vector of \(T\) on \(\partial T\) and
	\[
	G_T^p = R_T^m - S\frac{u_{ph}^m - u_p^m}{\tau_p} + \operatorname{div}(K\nabla u_{ph}^m) + \frac{\alpha_{ex}}{2\pi}(u_{ph}^m - u_{ph}^c)\delta_{\Gamma}.
	\]
	
	Let \(F\) be an interior face of the mesh \(\mathcal{T}_{ph}\), shared by two elements \(T_1\) and \(T_2\). The jump of the normal derivative of \(u_{ph}^m\) across \(F\) is defined as:
	\[
	\left[\frac{\partial u_{ph}^m}{\partial n_F}\right] = \nabla u_{ph}^m \cdot n_{T_1} + \nabla u_{ph}^m \cdot n_{T_2},
	\]
	where \(n_{T_1}\) and \(n_{T_2}\) denote the outward unit normals to \(F\) with respect to \(T_1\) and \(T_2\), respectively. We then define:
	\[
	g_F^p = -K\left[\frac{\partial u_{ph}^m}{\partial n_F}\right],
	\]
	for every interior face \(F \in \mathcal{F}_{ph}\).
	
	The second term of (\ref{63}) can be decomposed over all interior faces using the jump \(g_F^p\). Each interior face \(F\) is shared by two elements, and the contributions from both sides are added. Thus, for any \(w^m \in H_0^1(\Omega_m)\), we have:
	\begin{equation}
		A_p(e_p^m,w^m) = \tau_p \sum_{T\in \mathcal{T}_{ph}} \int_T G_T^p \cdot w^m + \tau_p \sum_{F\in \mathcal{F}_{\mathrm{int}}} \int_F g_F^p \cdot w^m + S(e_{p-1,h}^m,w^m) + \tau_p \alpha_{ex} \sum_{E\in \mathcal{P}_{ph}} \int_E e_p^c \Pi w^m. \quad \label{64}
	\end{equation}
	
	Moreover, for any \(w_{ph}^m \in W_{ph}^m\), we have:
	\begin{equation}
		A_p(e_p^m,w_{ph}^m) = S(e_{p-1,h}^m,w_{ph}^m) + \tau_p \alpha_{ex} \sum_{E\in \mathcal{P}_{ph}} \int_E e_p^c \Pi w_{ph}^m. \quad \label{65}
	\end{equation}
	
	\begin{lemma}[Galerkin Orthogonality]
		For any \(w_{ph}^m \in W_{ph}^m\), we have:
		\begin{equation}
			K\sum_{T\in \mathcal{T}_{ph}} \int_T \nabla e_p^m \cdot \nabla w_{ph}^m = -S\left(\frac{e_p^m - e_{p-1}^m}{\tau_p}, w_{ph}^m\right) - \frac{\alpha_{ex}}{2\pi} \sum_{T\in \mathcal{T}_{ph}} \int_T (e_p^m - e_p^c)\delta_{\Gamma} \cdot w_{ph}^m \quad \label{66}
		\end{equation}
	\end{lemma}
	
	\begin{lemma}\label{lem1}\mbox{ }\\
		$\bullet$ For any \(\phi \in H^1(\Omega_m)^3\), we have:
		\begin{equation}
			\int_{\Omega_m} \nabla e_p^m \cdot \mathbf{curl} \, \phi = \sum_{F\in \mathcal{F}_{ph}} \int_F J_{F,t}^p \cdot \phi. \quad \label{67}
		\end{equation}
		
		$\bullet$ For any \(\phi_{ph} \in (W_{ph}^m)^3\), we have:
		\begin{equation}
			\int_{\Omega_m} \nabla e_p^m \cdot \mathbf{curl} \, \phi_{ph} = 0. \quad \label{68}
		\end{equation}
	\end{lemma}
	\begin{proof}
		See \cite[Lemma 3.2 and Lemma 3.3]{nicaise2005posteriori}.
	\end{proof}
	
	\begin{lemma}
		For any \(w^m \in H_0^1(\Omega_m)\), we have:
		\begin{eqnarray}\nonumber
			K(\nabla e_p^m,\nabla w^m)& =& \left(R_p^m - S\frac{u_p^m - u_{p-1}^m}{\tau_p} - \frac{\alpha_{ex}}{2\pi}(u_p^m - u_p^c)\delta_{\Gamma}, w^m\right) \\
			&+& \sum_{T\in \mathcal{T}_{ph}} \int_T \operatorname{div}(K\nabla u_{ph}^m) w^m + K\sum_{F\in \mathcal{F}_{ph}^{\mathrm{int}}} \int_F J_{F,n}^m \cdot w^m. \quad \label{69}
		\end{eqnarray}
	\end{lemma}
	
	We recall the following result drawn from \cite{forste1987finite, nicaise2005posteriori}.
	
	\begin{lemma}[Helmholtz Decomposition of the Error in the Matrix]
		We have the following error decomposition:
		\begin{equation}
			\nabla e_p^m = \nabla v_p^m + \mathbf{curl} \, \phi_p^m, \quad \label{70}
		\end{equation}
		where $\phi_p^m \in [H^1(\Omega_m)]^3$. Moreover the next estimates hold:
		\begin{equation}
			\begin{array}{r}
				\| v_p^m\|_{L^2(\Omega_m)} \leq \| e_p^m\|_{H_0^1(\Omega_m)},\\
				\| \phi_p^m\|_{L^2(\Omega_m)} \leq \| e_p^m\|_{H_0^1(\Omega_m)}.
			\end{array} \quad \label{71}
		\end{equation}
	\end{lemma}
	
	\begin{lemma} (Residual equations)\label{lem2}
		For any \(w^m \in H_0^1(\Omega_m)\) and \(\phi^m \in (H_0^1(\Omega_m))^3\), we have the following equations:
		\begin{eqnarray}\label{73}
		\begin{array}{rl}
			& K(\nabla e_p^m,\nabla w^m) = \displaystyle\sum_{T\in \mathcal{T}_{ph}} \int_T G_T^p \cdot (w^m - I_{\mathrm{Cl}}^0 w^m) + K\sum_{F\in \mathcal{F}_{ph}} \int_T J_{F,n}^p \cdot (w^m - I_{\mathrm{Cl}}^0 w^m) \\
			& \qquad -\left(S\frac{e_p^m - e_{p-1}^m}{\tau_p} + \frac{\alpha_{ex}}{2\pi}(e_p^m - e_p^c)\delta_{\Gamma}, w^m\right), \\
			& \displaystyle\int_{\Omega_m} \nabla e_p^m \cdot \mathbf{curl} \, \phi^m = \sum_{E\in \mathcal{F}_{ph}} \int_F J_{F,t}^p \cdot (\phi^m - I_{\mathrm{Cl}}^0 \phi^m), \\
			& S\| e_p^m\|_{L^2(\Omega_m)}^2 + \tau_p K\| e_p^m\|_{H_0^1(\Omega_m)}^2 + \tau_p \frac{\alpha_{ex}}{2\pi}\displaystyle\int_{\Omega_m} (e_p^m)^2 \delta_{\Gamma} \\\label{74}
			& = S(e_{p-1}^m,e_p^m) + \tau_p \frac{\alpha_{ex}}{2\pi}\displaystyle\int_{\Omega_m} e_p^c v_p^m \delta_{\Gamma} - \tau_p \frac{\alpha_{ex}}{2\pi}\displaystyle\sum_{T\in \mathcal{T}_{ph}} \int_T e_p^c \delta_{\Gamma} \cdot I_{\mathrm{Cl}}^0(e_p^m - v_p^m) \\
			& + \left(S(e_p^m - e_{p-1}^m) + \tau_p \frac{\alpha_{ex}}{2\pi} e_p^m \delta_{\Gamma}, e_p^m - v_p^m - I_{\mathrm{Cl}}^0(e_p^m - v_p^m)\right) \\
			& + \tau_p K\displaystyle\sum_{T\in \mathcal{T}_{ph}} \displaystyle\int_T \nabla e_p^m \cdot \nabla I_{\mathrm{Cl}}^0(e_p^m - v_p^m) + \tau_p \sum_{T\in \mathcal{T}_{ph}} \int_T G_T^p \cdot (v_p^m - I_{\mathrm{Cl}}^0 v_p^m) \\
			& + \tau_p \displaystyle\sum_{F\in \mathcal{F}_{ph}} \int_F \left[ J_{F,n}^p \cdot (v_p^m - I_{\mathrm{Cl}}^0 v_p^m) + 
			J_{F,t}^p \cdot (\phi_p^m - I_{\mathrm{Cl}} \phi_p^m) \right]. \quad \label{75}
		\end{array}
		\end{eqnarray}
	\end{lemma}
	
	\begin{proof}
		Using the decomposition (\ref{70}), and setting \(w^p = v_p^m\) and \(\phi_p = \phi_p^m\) respectively in $(\ref{73})_1$ and $(\ref{74})_2$, we have:
		\[
		K\| e_p^m\|_{H_0^1(\Omega_m)}^2 = K(\nabla e_p^m,\nabla v_p^m) + K(\nabla e_p^m,\mathbf{curl} \, \phi_p^m)
		\]
		\[
		= -\left(S\frac{e_p^m - e_{p-1}^m}{\tau_p} + \frac{\alpha_{ex}}{2\pi}(e_p^m - e_p^c)\delta_{\Gamma}, v_p^m\right) + \sum_{T\in \mathcal{T}_{ph}} \int_T G_T^p \cdot (v_p^m - I_{\mathrm{Cl}}^0 v_p^m)
		\]
		\[
		+ \sum_{F\in \mathcal{F}_{ph}} \int_F \left[ f_{F,n}^p \cdot (v_p^m - I_{\mathrm{Cl}}^0 v_p^m) + f_{F,t}^p \cdot (\phi_p^m - I_{\mathrm{Cl}} \phi_p^m) \right]
		\]
		
		We can deduce the following:
		\[
		\begin{array}{rl}
			& S\| e_p^m\|_{L^2(\Omega_m)}^2 + \tau_P K\| e_p^m\|_{H_0^1(\Omega_m)}^2 + \tau_P \frac{\alpha_{ex}}{2\pi}\int_{\Omega_m} (e_p^m)^2 \delta_{\Gamma} \\[5pt]
			= & S(e_{p-1}^m,e_p^m) + \tau_P \frac{\alpha_{ex}}{2\pi}\int_{\Omega_m} e_p^c v_p^m \delta_{\Gamma} \\[5pt]
			& + \left(S(e_p^m - e_{p-1}^m) + \tau_P \frac{\alpha_{ex}}{2\pi} e_p^m \delta_{\Gamma}, e_p^m - v_p^m - I_{\mathrm{Cl}}^0(e_p^m - v_p^m)\right) \\[5pt]
			& - \left(S(e_p^m - e_{p-1}^m) + \tau_P \frac{\alpha_{ex}}{2\pi} e_p^m \delta_{\Gamma}, I_{\mathrm{Cl}}^0(e_p^m - v_p^m)\right) \\[5pt]
			& + \tau_P \sum_{T\in \mathcal{T}_{ph}} \int_T G_T^p \cdot (v_p^m - I_{\mathrm{Cl}}^0 v_p^m) \\[5pt]
			& + \tau_P \sum_{F\in \mathcal{F}_{ph}} \int_F \left[ f_{F,n}^p \cdot (v_p^m - I_{\mathrm{Cl}}^0 v_p^m) + f_{F,t}^p \cdot (\phi_p^m - I_{\mathrm{Cl}} \phi_p^m) \right] \quad \label{75'}
		\end{array}
		\]
		
		Using the Galerkin orthogonality relation (\ref{66}) and setting \(w_p^m = I_{\mathrm{Cl}}^0(e_p^m - v_p^m)\), we obtain (\ref{75}).
	\end{proof}
	
	As usual \cite{ainsworth2000posteriori}, the exact residuals \(G_T^p\) are replaced by an approximate element residuals. We denote respectively by \(\overline{G}_T^p\) the projection in convenient finite dimensional spaces of \(G_E^p\). This consists in taking the approximation of \(R_{ph}^m\) as projection of \(R_p^m\) in finite dimensional spaces since \(R_p^c\) is the only one that belongs to an infinite dimensional space. Usually, \(R_{ph}^m\) is considered as the mean value of \(R_p^m\), i.e
	\[
	R_{ph}^m = \frac{1}{|\Omega_c|}\int_{\Omega_c} R_p^m(x) dx.
	\]
	
	Similarly to \cite{nicaise2005posteriori}, we define the error estimators in the following.
	
	\begin{definition}
		Let \(p \geq 1\). The local error estimate \(\eta_T^p\) on the element \(T \in \mathcal{T}_{ph}\) is defined by:
		\[
		\eta_T^p = h_{\min,T} \| \overline{G}_T^p\|_E + \sum_{F\in \mathcal{F}_T} \frac{h_{\min,F}}{h_F^{1/2}} \left( \| \mathcal{J}_{F,n}\|_F + \| \mathcal{J}_{F,t}\|_F \right).
		\]
		
		The global spatial error estimator \(\eta_p^m\) in the porous matrix is given by:
		\[
		\eta_p^m = \left( \sum_{T\in \mathcal{T}_{ph}} (\eta_T^p)^2 \right)^{1/2}.
		\]
		
		The local and global approximation terms are defined by:
		\[
		\xi_T^p = h_{\min,T} \| R_p^m - R_{ph}^m\|_{\omega_T} \quad \text{and} \quad \xi_p^m = \left( \sum_{T\in \mathcal{T}_{ph}} (\xi_T^p)^2 \right)^{1/2}.
		\]
	\end{definition}
	We have the following result:
	\begin{theorem}
	For every $p\in[1,\ldots,N]$, 
	\begin{eqnarray}
		\parallel e_p^m\parallel_{H_0^1(\Omega_m)}\lesssim m_1(e_p^m,\mathcal{T}_{ph})(\eta_p^m+\zeta_p^m)+\parallel e_p^c\parallel_{H_0^1(a,b)}+
		\parallel e_p^m-e_{p-1}^m\parallel_{L^2(\Omega_m\parallel)}.
	\end{eqnarray}
	In particular, if the conduit error is controlled by Theorem \ref{G1}, then
	\begin{eqnarray}
		\parallel e_p^m\parallel_{H_0^1(\Omega_m)}\lesssim m_1(e_p^m,\mathcal{T}_{ph})(\eta_p^m+\zeta_p^m)+\eta_p^c+\zeta_p^c.
	\end{eqnarray}
	\end{theorem}
\begin{proof}
	From Lemma \ref{lem2}, choosing $v_p^m\in H_0^1(\Omega_m)$ and  using the Helmholtz decomposition $\nabla(e_p^m-v_p^m)=\curl \phi_p^m$, we obtain the residual identity 
	$$\mathcal{A}_p(e_p^m,e_p^m)=I_1+I_2+I_3+I_4+I_5,$$
	where the terms $I_i$ correspond respectively to the element residuals, face residuals, tangential jumps, temporal residuals and coupling residuals.\\\\
	\textbf{Estimate of the element residual term.} Using the Clément interpolation estimate (\ref{26}), Cauchy-Schwarz inequality and the definition of $\eta_p^m$ and $\zeta_p^m$, one obtains
	$$|I_1|\lesssim m_1(v_p^m,\mathcal{T}_{ph})(\eta_p^m+\zeta_p^m)\parallel v_p^m\parallel_{H_0^1(\Omega_m)}.$$
	Indeed, the manuscript already establishes 
	$$\displaystyle\sum_{T\in\mathcal{T}_{ph}} G_T^p(v_p^m-I_{Cl}^0 v_p^m)\lesssim m_1(v_p^m,\mathcal{T}_{ph})(\eta_p^m+\zeta_p^m)\parallel v_p^m\parallel_{H_0^1(\Omega_m)}.$$
	\textbf{Estimate of the face residual terms.}  Using  (\ref{27}) and (\ref{28}) together with Cauchy-Schwarz yields 
	$$|I_1|+|I_2|\lesssim m_1(v_p^m,\mathcal{T}_{ph})(\eta_p^m+\zeta_p^m)\parallel v_p^m\parallel_{H_0^1(\Omega_m)}.$$
	\textbf{Estimate of the temporal residual.} Using Cauchy-Schwarz, 
	$$|I_4|\lesssim \parallel S(e_p^m-e_{p-1}^m)+\tau_p\frac{\alpha_{ex}}{2\pi}e_p^m\delta_{\Gamma}\parallel_{H^{-1}(\Omega_m)}\parallel e_p^m-v_p^m\parallel_{H_0^1(\Omega_m)}.$$
	The boundedness of the trace operator and Proposition \ref{a1} imply
	$$|I_4|\lesssim \left(\parallel e_p^m-e_{p-1}^m\parallel_{L^2(\Omega_m)}+\parallel e_p^m\parallel_{H_0^1(\Omega_m)}\right)\parallel e_p^m-v_p^m\parallel_{H_0^1(\Omega_m)}.$$
	\textbf{Estimate of the coupling term.} For the term involving $e_p^c$,
	$$|I_5|=\left|\tau_p \frac{\alpha_{ex}}{2\pi}\int_{\Omega_m} e_p^c \delta_{\Gamma} I_{Cl}^0 (e_p^m-v_p^m)\right|,$$
	the trace inequality and the boundedness of $\Pi$ give
	$$|I_5|\lesssim \parallel e_p^c\parallel_{H_0^1(a,b)}\parallel e_p^m-v_p^m\parallel_{H_0^1(\Omega_m)}.$$
	Taking $v_p^m=e_p^m$, all interpolation terms vanish and the previous estimates leads to
	\begin{eqnarray*}
	\mathcal{A}_p(e_p^m,e_p^m)&\lesssim& m_1(e_p^m,\mathcal{T}_{ph})(\eta_p^m+\zeta_p^m)\parallel e_p^m\parallel_{H_0^1(\Omega_m)}\\
	&+&
	\left(\parallel e_p^c\parallel_{H_0^1(a,b)}+\parallel e_p^m-e_{p-1}^m\parallel_{L^2(\Omega_m)}
	\parallel e_p^m\parallel_{H_0^1(\Omega_m)}.
	\right)
	\end{eqnarray*}
Using the coercivity of $\mathcal{A}_p$, 
$$\mathcal{A}_p(e_p^m,e_p^m)\geq \alpha_p \parallel e_p^m\parallel_{H_0^1(\Omega_m)}^2,$$
and dividing by $\parallel e_p^m\parallel_{H_0^1(\Omega_m)}$, we obtain
$$\parallel e_p^m\parallel_{H_0^1(\Omega_m)}\lesssim m_1(e_p^m,\mathcal{T}_{ph})(\eta_p^m+\zeta_p^m)+
\parallel e_p^c\parallel_{H_0^1(a,b)}+\parallel e_p^m-e_{p-1}^m\parallel_{L^2(\Omega_m)}.
$$
This completes the proof.
\end{proof}

\section{Conclusion}\label{conclusion}
	We have investigated an a-priori and a-posteriori error analysis for the time-dependent case of a coupled continuum pipe-flow in three dimensions on anisotropic meshes (i.e. elements with very large aspect ratio). 
Our analyses include several nonconforming finite element methods:
various of the Crouzeix-Raviart element methods. These methods require the stability of the element
pairs in orders to guarantees existence and uniqueness of the discrete solution of the model. Leveraging
residual elements, local error indicators and a global estimators hold unconditionally while their
reliability depends on the alignment of the mesh in the porous matrix with the anisotropic solution.
Mesh requirements are established to ensure the reliability of the error estimator.
	\section*{Declaration}

	\subsection*{Conflict of interest} The authors declare that there is no conflict of interest regarding the publication of
	this paper.
	\subsection*{Funding} No funding
	
	\subsection*{Acknowledgments}The authors would like to thank the anonymous reviewers for their valuable comments
	and constructive suggestions, which significantly improved the clarity of the methodology and the
	presentation of the results.

	\appendix
	
	\section{Boundedness of the operator \(\Pi\)}\label{apin}
	
	\begin{proposition} \label{a1}
		There exists a constant \(C_{\Pi} > 0\), for any \(v \in H_0^1(\Omega_m)\) such that:
		\begin{equation}
			\| \Pi v\|_{L^2(a,b)} \leq C_{\Pi} \| u\|_{H_0^1(\Omega_m)}, \quad 
		\end{equation}
		
		Furthermore, for any \(v \in H_0^1(\Omega_m)\) and \(w \in H_0^1(\Omega_m)\), we have:
		\begin{equation}\label{a2}
			\left| \int_a^b \Pi (v w)(x) dx \right| \leq C_{\Pi}^2 \| v\|_{H_0^1(\Omega_m)} \| w\|_{H_0^1(\Omega_m)}. 
		\end{equation}
	\end{proposition}
	
	\begin{proof}
		Consider \(v \in H_0^1(\Omega_m)\). Using Cauchy-Schwarz inequality, for any \(x \in (a,b)\), we have:
		\[
		|\Pi u(x)|^2 \leq \frac{1}{2\pi} \int_0^{2\pi} |u(s_0(x,\theta))|^2 d\theta.
		\]
		
		Referring to the definition of \(s_0\) and \(\Gamma\), we can deduce:
		\[
		\| \Pi u\|_{L^2(a,b)}^2 \leq \frac{1}{2\pi r} \int_{\Omega_m} |u(x)|^2 dx.
		\]
		
		Denoting by \(C_{\Omega_m}\) the Poincaré constant relative to the domain \(\Omega_m\), we have:
		\[
		\| \Pi u\|_{L^2(a,b)}^2 \leq \frac{C_{\Omega_m}^2}{2\pi r} \| u\|_{H_0^1(\Omega_m)}^2.
		\]
		
		Taking \(C_{\Pi} = \frac{C_{\Omega_m}}{\sqrt{2\pi r}}\) leads to (A.1).
		
		Now, let's consider \(v\) and \(w\) in \(H_0^1(\Omega_m)\). Referring to the definition of \(s_0\) and \(\Gamma\), thanks to Cauchy-Schwarz inequality, for any \(x \in (a,b)\), we obtain:
		\[
		\left| \int_a^b \Pi (u v)(x) dx \right| \leq \frac{1}{2\pi r} \left( \int_{\Gamma} u^2 \right)^{1/2} \left( \int_{\Gamma} v^2 \right)^{1/2}.
		\]
		
		Applying Poincaré inequality to the right-hand term leads to (\ref{a1}).
	\end{proof}
	
	\section{Anisotropic vectors}
	
			\begin{figure}[!h]
			\centering
			\includegraphics[width=15cm,height=13cm]{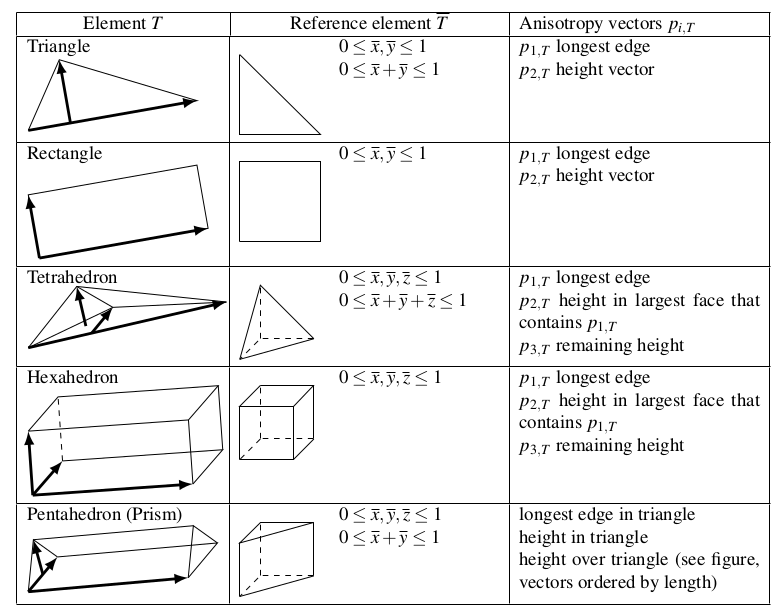}
		\end{figure}
\end{document}